\documentclass[american,10pt]{amsart}
\pdfoutput=1
\usepackage[T1]{fontenc}
\usepackage[utf8]{inputenc}
\usepackage{lmodern}
\usepackage{enumitem}
\usepackage{mathtools}
\usepackage{amssymb}
\usepackage{bm}
\usepackage[unicode=true,pdfusetitle,
 bookmarks=true,bookmarksnumbered=false,bookmarksopen=false,
 breaklinks=false,pdfborder={0 0 1},backref=false,colorlinks=false]{hyperref}
\usepackage[foot]{amsaddr}
\usepackage{cite}

\theoremstyle{plain}
\newtheorem{thm}{Theorem}[section]
\newtheorem{lem}[thm]{Lemma}
\newtheorem{prop}[thm]{Proposition}
\newtheorem{cor}[thm]{Corollary}
\newtheorem{conj}[thm]{Conjecture}

\theoremstyle{definition}
\newtheorem{defi}{Definition}[section]

\theoremstyle{remark}
\newtheorem{rem}{Remark}[section]
\newtheorem*{rem*}{Remark}

\newcommand\R{\mathbb{R}}
\newcommand\Sn{\mathbb{S}^{n-1}}
\newcommand\N{\mathbb{N}}
\newcommand\Z{\mathbb{Z}}
\newcommand\Q{\mathbb{Q}}
\DeclareMathOperator{\vspan}{span}
\DeclareMathOperator{\diag}{diag}
\newcommand\upd{\textup{d}}
\newcommand\upT{\textup{T}}
\newcommand\upp{\textup{per}}
\newcommand\upl{\textup{loc}}
\newcommand\upe{\textup{e}}
\newcommand\upDir{\textup{Dir}}

\newcommand\upini{\textup{ini}}

\newcommand\tl{\triangleleft}

\newcommand{\vect}[1]{\bm{\mathbf{#1}}} 
\newcommand\veL{\vect{L}} 
\newcommand\vef{\vect{f}} 
\newcommand\veB{\vect{B}} 
\newcommand\veu{\vect{u}} 
\newcommand\vev{\vect{v}} 
\newcommand\vew{\vect{w}} 
\newcommand\vep{\vect{p}} 
\newcommand\ver{\vect{r}} 
\newcommand\vez{\vect{0}} 
\newcommand\veo{\vect{1}} 
\newcommand\vei{\vect{\infty}} 
\newcommand\caP{\mathcal{P}} 
\newcommand\caC{\mathcal{C}} 
\newcommand\cbQ{\vect{\mathcal{Q}}} 
\newcommand\cbR{\vect{\mathcal{R}}} 
\newcommand\dcbP{\diag(\vect{\caP})} 
\newcommand\clOmper{[0,1]^{1+n}} 
\newcommand\FPH{F\"{o}ldes--Pol\'{a}\v{c}ik's Harnack } 
\newcommand\muwed{\mu^{\wedge}} 
\newcommand\muvee{\mu^{\vee}} 

\usepackage{todonotes}
\newcommand{\revision}[1]{{#1}}

\title[Pulsating traveling wave solutions of KPP systems in periodic media]{Planar pulsating traveling wave solutions of non-cooperative Fisher--KPP systems in space-time periodic media}
\author{L\'{e}o Girardin}
\address[L. G.]{CNRS, Institut Camille Jordan, Universit\'{e} Claude Bernard Lyon-1, 43 boulevard du 11 novembre 1918, 69622 Villeurbanne Cedex, France}
\email{leo.girardin@math.cnrs.fr}
\author{Gr\'{e}goire Nadin}
\address[G. N.]{CNRS, Institut Denis Poisson, Universit\'{e} d'Orl\'{e}ans, Rue de Chartres, 45067 Orl\'{e}ans Cedex 2, France}
\email{gregoire.nadin@cnrs.fr}

\begin{document}
\begin{abstract}
    Non-cooperative Fisher--KPP systems with space-time periodic coefficients are motivated
    for instance by models for structured populations evolving in periodic environments.
    This paper is concerned with entire solutions describing the invasion of open space by a 
    persistent population at constant speed. These solutions are important in the understanding of 
    long-time behaviors for the Cauchy problem. Adapting methods developed for scalar equations 
    satisfying the comparison principle as well as methods developed for systems with homogeneous 
    coefficients, we prove, in each spatial direction, the existence of a critical speed such that:
    there exists no almost planar generalized transition waves with a smaller speed;
    if the direction is rational, each rational speed not smaller than the critical speed
    is the speed of a \revision{planar pulsating traveling wave with time and transverse space periodicity};
    if the coefficients are homogeneous in space, each speed not smaller than the critical speed
    is the speed of a \revision{planar pulsating traveling wave with time periodicity}.
\end{abstract}

\maketitle

\section{Introduction}

\subsection{General framework}
This paper is concerned with reaction--diffusion systems of the form 
\begin{equation}\label{sys:KPP}\tag{KPP}
    \dcbP\veu =\veL\veu-(\veB\veu)\circ\veu,
\end{equation}
where: $\veu:\R\times\R^n\to\R^N$ is a vector-valued function of size $N\in\N^\star$, 
with a time variable $t\in\R$ and a space variable $x\in\R^n$, 
$n\in\N^\star=\N\backslash\{0\}$ being the spatial dimension; 
each operator of the family $\vect{\caP}=(\caP_i)_{i\in[N]}$, with $[N]=\N\cap[1,N]$, has the form
\[
    \caP_i:u\mapsto\partial_t u -\nabla\cdot\left(A_i\nabla u\right)+q_i\cdot\nabla u,
\]
with $A_i:\R\times\R^n\to\R^{n\times n}$ and $q_i:\R\times\R^n\to\R^n$ periodic functions of 
$(t,x)$, respectively square matrix-valued and vector-valued;
$\veL,\veB:\R\times\R^n\to\R^{N\times N}$ are square matrix-valued periodic functions of
$(t,x)$; and $\circ$ denotes the Hadamard product between two vectors in $\R^N$.

The standing assumptions on $\vect{\caP}$, $\veL$ and $\veB$ are the following.
\begin{enumerate}[label=$({\mathsf{A}}_{\arabic*})$]
    \item \label{ass:ellipticity} The family $(A_i)_{i\in[N]}$ is \textit{uniformly elliptic}:
    \[
        0<\min_{i\in[N]}\min_{y\in\Sn}\min_{(t,x)\in\R\times\R^n}\left(y\cdot A_i(t,x)y\right).
    \]
    \item \label{ass:cooperative} The matrix $\underline{\veL}\in\R^{N\times N}$, whose entries are 
    \[
        \underline{l}_{i,j}=\min_{(t,x)\in\R\times\R^n}l_{i,j}(t,x)\quad\text{for all }(i,j)\in[N]^2,
    \]
    is \textit{essentially nonnegative}: its off-diagonal entries are nonnegative.
    \item \label{ass:irreducible} The matrix $\overline{\veL}\in\R^{N\times N}$, whose entries are 
    \[
        \overline{l}_{i,j}=\max_{(t,x)\in\R\times\R^n}l_{i,j}(t,x)\quad\text{for all }(i,j)\in[N]^2,
    \]
    is \textit{irreducible}: it does not have a stable subspace of the form
    $\vspan(\vect{e}_{i_1},\dots,\vect{e}_{i_k})$, where $k\in[N-1]$, $i_1,\dots,i_k\in[N]$ and $\vect{e}_i=(\delta_{ij})_{j\in[N]}$.
    By convention, $[0]=\emptyset$ and $1\times 1$ matrices are irreducible, even if zero.
    \item \label{ass:KPP} The matrix $\underline{\veB}\in\R^{N\times N}$, whose entries are 
    \[
        \underline{b}_{i,j}=\min_{(t,x)\in\R\times\R^n}b_{i,j}(t,x)\quad\text{for all }(i,j)\in[N]^2,
    \]
    is \textit{positive}: its entries are positive.
    \item \label{ass:smooth_periodic} There exists $\delta\in(0,1)$ such that
	$\veL,\veB\in\caC^{\delta/2,\delta}_\upp(\R\times\R^n,\R^{N\times N})$ and, 
	for each $i\in[N]$, $A_i\in\caC^{\delta/2,1+\delta}_\upp(\R\times\R^n,\R^{n\times n})$ and
    $q_i\in\caC^{\delta/2,\delta}_\upp(\R\times\R^n,\R^n)$. 
    Moreover, $A_i=A_i^\upT$ for each $i\in[N]$.
\end{enumerate}
The precise definition of the functional spaces appearing in \ref{ass:smooth_periodic} will
be clarified below, if not clear already. As usual in such a smooth and 
generic framework, the symmetry of the diffusion matrices in \ref{ass:smooth_periodic} is actually given for free.
No symmetry assumption is made on $\veL$. \revision{Another consequence of the smoothness
of the coefficients is that} the irreducibility of $\overline{\veL}$ in \ref{ass:irreducible}
is equivalent to the irreducibility of the space-time average of $\veL$.

In this paper, we are interested in solutions defined in $\R\times\R^n$, referred to as entire solutions.
More precisely, we will seek classical entire solutions modeling invasions of open space at constant speed in a given direction.

\subsection{Organization of the paper}
The remainder of Section 1 is devoted to a detailed introduction and to the statement of the main results. 
Section 2 contains technical preliminaries. Sections 3, 4 and 5 contain the proofs.

\subsection{State of the art}

The notion of \textit{planar traveling waves} has first been investigated in the pioneering papers \cite{Fisher_1937, KPP_1937}
in the framework of homogeneous scalar Fisher--KPP equations ($N=1$ and the coefficients are independent of space and time). 
Such solutions are entire solutions that can have the form $u(t,x)=p(x\cdot e+ct)$, with $p>0$, $p(-\infty)=0$ and $p(+\infty)=1$
(where $1$ is the unique positive stationary solution of the scalar equation). 
Here, $-e\in\Sn$ is referred to as the direction of propagation, $c$ as the speed of the traveling wave, and $p$ as its profile. 
Such solutions are important to describe the attractor of the Cauchy problem depending on the form of the initial condition; 
in particular, the one with minimal speed attracts, in a sense, the solutions of the Cauchy problem associated with compactly 
supported initial values.

When the coefficients of the scalar Fisher--KPP equation depend periodically on space and remain homogeneous in time, 
or vice versa, a new notion is needed in order to encapsulate this periodicity. 
\revision{In \cite{Shigesada_Kawasaki_Teramoto_1986}, where periodic media were investigated for the first time,
in the framework of space periodic coefficients in a one-dimensional space,
a solution $u$ was referred to as a traveling periodic wave if it satisfied $u(t,x)=u\left(t+\frac{L}{c},x+L\right)$, 
where $L$ was the period of the coefficients and $c$ the wave speed.
Later, a compatible functional form $u(t,x)=p(x\cdot e+ct,x)$ was introduced by Xin in \cite{Xin_1991,Xin_1991_2} for scalar 
bistable and combustion reaction--diffusion equations in arbitrary spatial dimension. In this functional form, $p>0$, $p(z,\cdot)$ 
is periodic for all $z\in \R$, and the two limits $p(-\infty,\cdot)$ and $p(+\infty,\cdot)$ are the two steady states connected
by the wave. In \cite{Xin_1991,Xin_1991_2}, such solutions were simply referred to as traveling waves in periodic media. 
This functional form is robust and can be used for other scalar reaction--diffusion equations with space periodic coefficients, 
including Fisher--KPP equations. It can also be adapted for time periodic coefficients, where it 
becomes $u(t,x)=p(x\cdot e+ct,t)$ and where the limiting steady states become time periodic entire solutions.
Other names for such solutions have also been used in the literature, pulsating traveling front and pulsating traveling wave 
in particular. In this paper, for clarity, we only use the terminology \textit{planar pulsating traveling wave}.}
The existence of such solutions for scalar Fisher--KPP equations has been first derived in 
\cite{Hudson_Zinner_1995,Berestycki_Hamel_2002,Berestycki_Hamel_Roques_2} (space periodic case) and in 
\cite{FrejacquesPhD} (time periodic case).

When the coefficients are periodic in space and in time, scalar Fisher--KPP planar pulsating traveling waves now take 
the form $u(t,x)=p(x\cdot e+ct,t,x)$, where $(t,x)\mapsto p(z,t,x)$ is periodic for all $z$, $z\mapsto p(z,t,x)$ is 
increasing for all $(t,x)$, $p(-\infty,\cdot,\cdot)=0$ and $p(+\infty,\cdot,\cdot)$ is the unique uniformly positive 
entire solution \cite{Nadin_2009,Nolen_Rudd_Xin}. The construction of such solutions requires some care since they are
somehow the trace of a function $p(z,t,x)$ over the plane $z=x\cdot e+ct$.

The usual tool when studying entire solutions in scalar reaction--diffusion models is the \textit{comparison principle}, 
namely the property that ordered initial conditions yield perpetually ordered solutions. 
It is well known that the comparison principle does not necessarily hold for systems. 
Systems that satisfy the comparison principle are referred to as \textit{monotone} or \textit{cooperative}; 
consistently, systems that do not satisfy the comparison principle are referred to as \textit{non-monotone} or \textit{non-cooperative}.
A weakly coupled system of size $N$ is monotone if it satisfies the so-called \textit{Kamke condition}: 
for each $i\in[N]$, the $i$-th reaction term is nondecreasing with respect to each $u_j$, $j\in[N]\backslash\{i\}$
\cite{Cantrell_Cosner_03,Volpert_Volpert_Volpert}.

For monotone KPP-type systems (\textit{e.g.}, \eqref{sys:KPP} but with $\veB=\vect{I}$), 
arguments relying upon the comparison principle have proved
the existence of traveling waves when the coefficients are homogeneous \cite{Li_Weinberger_,Volpert_Volpert_Volpert} 
or of pulsating traveling waves when the coefficients are periodic \cite{Weinberger_200,Fang_Yu_Zhao,Du_Li_Shen_2022}. 

However, due to the assumption \ref{ass:KPP}, \eqref{sys:KPP} is \revision{globally} monotone if and only if $N=1$.
When $N\geq 2$, it is known that $N$, $\veL$ and $\veB$ can be chosen in such a way that a comparison principle holds true
in a neighborhood of $\vez$ or away from it \cite[Proposition 2.5]{Cantrell_Cosner_Yu_2018}. But such choices are not generic 
and never lead to a global comparison principle.

For non-cooperative systems, a variational approach is sometimes possible when the reaction term has a gradient structure and $\dcbP=\partial_t-\Delta$ \cite{Muratov_2004,Risler_2006,Muratov_Novaga_2008}. The system \eqref{sys:KPP} does not satisfy these assumptions in general. Even with $\dcbP=\partial_t-\Delta$ and $\veL$ symmetric, the term $-\veB\veu\circ\veu$ does not have a gradient structure. 

New arguments are therefore needed in order to derive the existence of (pulsating) traveling waves for \eqref{sys:KPP} when $N>1$. 
Such existence results have been derived before in the following special cases:
\begin{itemize}
    \item $N=2$, $n=1$, no advection and space-time homogeneous coefficients \cite{Griette_Raoul,Morris_Borger_Crooks};
    \item $n=1$, no advection, space-time homogeneous coefficients and pointwise essentially positive $\veL$ \cite{Wang_2011};
    \item $n=1$, no advection and space-time homogeneous coefficients \cite{Girardin_2016_2};
    \item $N=2$, $n=1$, $A_1=A_2=1$, no advection, space periodic coefficients and pointwise essentially positive $\veL$ \cite{Alfaro_Griette};
    \item $N=2$, $n=1$, space periodic coefficients and pointwise essentially positive $\veL$ \cite{Griette_Matano_2021}.
\end{itemize}

In this paper, we present a unifying approach for non-cooperative KPP systems of arbitrary size $N\in\N^\star$, in 
arbitrary spatial dimension $n\in\N^\star$, with space-time periodic coefficients, with advection
and with minimal positivity requirements on $\veL$. 
This was raised as an open question in a previous work by the first author \cite{Girardin_2023}.
To the best of our knowledge, the present paper is the very first study on invasions of open space modeled
by entire solutions of reaction--diffusion non-cooperative systems with space-time periodic coefficients. 

\subsection{Motivations}

Systems of the form \eqref{sys:KPP} arise naturally in population dynamics modeling when the population
under consideration has to be divided into discrete classes. Extensive references and discussions on these models can be
found in \cite{Girardin_2016_2,Girardin_Mazari_2022}. Here we only suggest briefly three examples of application.

Alfaro--Griette \cite{Alfaro_Griette} proved the existence of a pulsating traveling wave for the following system:
\begin{equation*}
    \begin{dcases}
	\partial_t u = \partial_{xx}u + u(r_u(x)-\gamma_u(x)(u+v)) + \mu(x) v - \mu(x) u, \\
	\partial_t v = \partial_{xx}v + v(r_v(x)-\gamma_v(x)(u+v)) + \mu(x) u - \mu(x) v.
    \end{dcases}
\end{equation*}
This system was conceived as a model in evolutionary epidemiology for spatio-temporal dynamics,
in a one-dimensional spatially periodic medium, of a pathogen population with two phenotypes 
$u$ and $v$. 
Each phenotype $i\in\{u,v\}$ has a growth rate $r_i=r_i(x)$, is subjected to Lotka--Volterra 
competition exerted by the total population $u+v$ at a rate $\gamma_i=\gamma_i(x)$, 
and mutates into the other phenotype at rate $\mu=\mu(x)$. 
All coefficients are uniformly positive. The one-dimensional spatial periodicity was presented as a first approach
of spatial heterogeneities, which are known to influence greatly host--parasite systems.

In a theoretical evolutionary biology article, Ress et al. \cite{Ress_Traulsen_Pichugin_2022} analyzed the following ODE model:
\begin{equation*}
    \begin{dcases}
	\displaystyle\frac{\upd u_1}{\upd t} = -b_1 u_1 - d_1 u_1 + N b_N\pi_1 u_N - u_1\sum_{j=1}^Nk_{1,j}u_j, \\
	\displaystyle\frac{\upd u_i}{\upd t} = (i-1)b_{i-1}u_{i-1} -i b_i u_i - d_i u_i + N b_N\pi_i u_N - u_i\sum_{j=1}^Nk_{i,j}u_j & \text{for }2\leq i\leq N. 
    \end{dcases}
\end{equation*}
This is a model for the population dynamics of cell groups that grow in size and fragment, giving rise to new
cell groups. The cell group densities $u_i$ contain groups of $i$ cells and are characterized by division rates $b_i>0$ (thus the
cell group grows at rate $i b_i$), death rates $d_i>0$, fragmentation rates $\pi_i\geq 0$, and competition rates $k_{i,j}>0$. 
The cell group density $u_N$ contains groups with the maximal number of cells. 
When a cell group reaches after a division the size $N+1$, it immediately fragments and gives rise to $\pi_1$
groups of size $1$, $\pi_2$ groups of size $2$, etc., so that $\sum_{i=1}^N i\pi_i=N+1$.

Gueguezo et al. \cite{Gueguezo_Doumate_Salako_2024} investigated the following model in a bounded spatial domain
$\Omega\subset\R^n$ with homogeneous Neumann boundary conditions:
\begin{equation*}
    \begin{dcases}
	\partial_t u_1 = d_1\Delta u_1 + r(t,x)u_2-s(t,x)u_1-(a(t,x)+b(t,x)u_1+c(t,x)u_2)u_1, \\
	\partial_t u_2 = d_2\Delta u_2 + s(t,x)u_1 - (e(t,x)+f(t,x)u_2+g(t,x)u_1)u_2. \\
    \end{dcases}
\end{equation*}
This system models the dynamics of a stage-structured population with juveniles $u_1$ and adults $u_2$.
The parameters $r,s,a,b,c,e,f,g$ are positive periodic functions of time and heterogeneous functions of space. In particular,
$r$ is the birth rate of juveniles and $s$ is the maturation rate measuring the transition between the juvenile stage
and the adult stage.

Our framework makes it possible to complexify these three models in various biologically meaningful ways and 
to consider invasion waves. 

We especially point out that, in the second example, the corresponding matrix $\veL$ is as follows:
\begin{equation*}
    \veL=
    \begin{pmatrix}
	-b_1-d_1 & 0 & \dots & 0 & Nb_N\pi_1 \\
	b_1 & -2b_2-d_2 & \dots & 0 & Nb_N\pi_2 \\
	0 & 2b_2 & \dots & 0 & Nb_N\pi_3 \\
	\vdots & \vdots & \ddots & \ddots & \vdots \\
	0 & 0 & \dots & (N-1)b_{N-1} & Nb_N\pi_N - Nb_N - d_N
    \end{pmatrix}
\end{equation*}
It is neither symmetric nor essentially positive. It is irreducible if $\pi_1>0$ and reducible otherwise.
If, for instance, coefficients are time periodic due to external stress and fragmentations give rise to groups of size $1$
during the first half-period but not during the second half-period, then $\veL$ is not always irreducible but its time average
is irreducible indeed. This exemplifies the relevancy of the assumption \ref{ass:irreducible}.

\subsection{Nondimensionalization}

Let $(T,L_1,\dots,L_n)$ be a family of $n+1$ positive numbers and denote temporarily by $\widetilde{\cdot}$ the composition 
with the diffeomorphism $(t,x)\mapsto (Tt, L_1 x_1, \dots, L_n x_n)$. For any entire solution $\veu$ of \eqref{sys:KPP},
\begin{equation*}
    \begin{split}
	\frac{1}{T}\partial_t\widetilde{\veu} = & \ \diag\left( \diag\left( \frac{1}{L_\alpha} \right)_{\alpha\in[n]}\nabla\cdot \left( \widetilde{A_i}\diag\left( \frac{1}{L_\alpha} \right)_{\alpha\in[n]}\nabla \right) \right)\widetilde{\veu} \\
	& \ -\diag\left( \widetilde{q_i}\cdot\diag\left( \frac{1}{L_\alpha} \right)_{\alpha\in[n]}\nabla \right)\widetilde{\veu} 
	+\widetilde{\veL}\widetilde{\veu}-(\widetilde{\veB}\widetilde{\veu})\circ\widetilde{\veu}.
    \end{split}
\end{equation*}
Therefore, up to redefining the coefficients of the system in a way that preserves \ref{ass:ellipticity}--\ref{ass:smooth_periodic},
we assume without loss of generality and from now on that the coefficients of \eqref{sys:KPP} are $1$-periodic 
with respect to each temporal or spatial variable.

\subsection{Entire solutions and invasions of open space}

It is natural to introduce a \textit{wave direction} $-e\in\Sn$ (or direction for short) and a
\textit{wave speed} $c\in\R$ (or speed for short), and to consider the change of variable corresponding to the moving
frame. Denoting by $\cdot^{\tl}$ the composition with the diffeomorphism $(t,x)\mapsto(t,x-cte)$, an
entire solution $\veu$ of \eqref{sys:KPP} satisfies:
\begin{equation}
    \partial_t\veu^{\tl}-\diag\left(\nabla\cdot\left(A^{\tl}_i\nabla\right)\right)\veu^{\tl}+\diag\left(\left(q^{\tl}_i+ce\right)\cdot\nabla\right)\veu^{\tl} = \veL^{\tl}\veu^{\tl}-(\veB^{\tl}\veu^{\tl})\circ\veu^{\tl}.
    \label{sys:KPP_moving_frame}\tag{$\text{KPP}^{\tl}$}
\end{equation}

\subsubsection{Almost planar generalized transition waves}

A first class of entire solutions, relevant to describe the planar invasion of $\vez$ by positive population densities,
can be defined as follows, in the spirit of \cite{Berestycki_Hamel_2012}, but taking into
account that the asymptotics as $x\cdot e\to+\infty$ are unclear and at least not unique
(refer for instance to \cite{Girardin_2017,Morris_Borger_Crooks} for examples with two stable positive 
steady states).

\begin{defi}\label{defi:GTW}
    A classical entire solution $\veu$ of \eqref{sys:KPP} is an \textit{almost planar generalized transition wave 
    in direction $-e\in\Sn$ with global mean speed $c\in\R$} if:
    \begin{enumerate}
	\item $\veu$ is nonnegative;
	\item $\veu$ is globally bounded in $\R\times\R^n$;
	\item $\veu^{\tl}:(t,x)\mapsto\veu(t,x-cte)$ satisfies the following two limits:
	    \begin{equation}
		\lim_{z\to-\infty}\sup_{(t,x)\in\R\times e^\perp}\max_{i\in[N]}u^{\tl}_i(t,x+ze)=0,
		\label{eq:definition_GTW_limit_negative_infinity}\tag{downstream}
	    \end{equation}
	    \begin{equation}
		\liminf_{z\to+\infty}\inf_{(t,x)\in\R\times e^\perp}\min_{i\in[N]}u^{\tl}_i(t,x+ze)>0.
		\label{eq:definition_GTW_limit_positive_infinity}\tag{upstream}
	    \end{equation}
    \end{enumerate}
\end{defi}

However this notion of solution is in the present context unsatisfying; the space-time periodicity of the coefficients 
of \eqref{sys:KPP} should be somehow inherited by the profile $\veu^{\tl}$ of the wave. We are
therefore seeking almost planar generalized transition waves that satisfy in addition
some form of periodicity.

\subsubsection{Planar pulsating traveling waves}\label{sec:PPTW}

This leads to the notion of planar pulsating traveling wave solution in space-time periodic media. 
For consistency, it should contain as particular cases several notions already investigated in the literature and mentioned above:
\begin{enumerate}
    \item if $N=1$, the notion of scalar pulsating traveling wave solution when coefficients
	are space-time periodic \cite{Nolen_Rudd_Xin,Nadin_2009};
    \item if $N=2$ and $n=1$, the notion of pulsating traveling wave solution when coefficients are 
	space periodic and time homogeneous \cite{Alfaro_Griette,Griette_Matano_2021};
    \item if $N\geq 2$, $n=1$ and coefficients are space-time homogeneous, the notion of (classical)
	traveling wave solution \cite{Griette_Raoul,Morris_Borger_Crooks,Girardin_2016_2,Wang_2011}.
\end{enumerate}

In particular, we recall that a scalar pulsating traveling wave solution 
is an almost planar generalized transition wave in direction $-e$ with global mean speed $c$ and 
such that $u(t,x)=p(x\cdot e +ct,t,x)$, with a profile $p$ which is increasing in its first 
variable and periodic in its last two variables. 

However, in view of the space-time homogeneous case, some monotonicity
loss with respect to the first variable seems inevitable as soon as $N\geq 2$ (\textit{cf.}
\cite{Griette_Raoul} where a non-monotonic profile is exhibited for $N=2$); at best, the 
requirement could be that $\vep$ is monotonic in its first variable in a left half-line 
\cite{Girardin_2016_2}. This leads to the following candidate.

\begin{defi}\label{defi:PTW_whole_space}
    A classical entire solution $\veu$ of \eqref{sys:KPP} is a 
    \revision{\textit{planar pulsating traveling wave in direction $-e\in\Sn$ at speed $c\in\R$ with time and space periodicity}} if:
    \begin{enumerate}
	\item it is an almost planar generalized transition wave in direction $-e$ with global mean speed $c$;
	\item it has the form $\veu(t,x)=\vep(x\cdot e +ct,t,x)$;
	\item with respect to its second and third variable, the wave profile $\vep$ is $1$-periodic.
    \end{enumerate}
\end{defi}

Interestingly, this definition could also encompass the notion of point-to-periodic traveling
wave solution, useful at least for systems with $N\geq 3$, space-time homogeneous coefficients, and
Hopf bifurcations \cite{Girardin_2018}.

Unfortunately, it turns out that the monotonicity loss in a right half-line breaks the known 
method of proof of existence \cite{Nadin_2009,Nolen_Rudd_Xin}. 
\revision{Indeed this method strongly relies on the equality 
$\int_{\R}|\partial_z p(z,t,x)|\upd z = \int_{\R}\partial_z p(z,t,x)\upd z$.}
\revision{Although it could be tempting to restrict the search to solutions that are monotonic in the variable $z$,
or more generally to solutions such that $\partial_z p(\cdot,t,x)$ is integrable in $\R$ uniformly in $(t,x)$, this would fail to capture
many solutions that can be easily observed with numerical simulations, for instance ``point-to-periodic'' traveling waves
$\veu(t,x)=\vep(x\cdot e+ct)$ that connect $\vez$ to a space periodic steady solution.}
Despite substantial efforts, we did not manage to circumvent this obstacle. 
Hence \revision{this paper cannot prove that Definition \ref{defi:PTW_whole_space} does not define an empty set
when $N\geq 2$ and the coefficients of \eqref{sys:KPP} are neither time homogeneous nor
space homogeneous, which is really the problem we have in mind}. 

Instead, we will exploit a weaker notion, inspired in particular by the definition of 
\textit{almost pulsating wave} introduced in \cite{Fang_Yu_Zhao} for monotone systems in 
one-dimensional spaces.

By virtue of the space-time periodicity of the coefficients of \eqref{sys:KPP}, if $e\in\Q^n$ and $c\in\Q$, 
then the coefficients of \eqref{sys:KPP_moving_frame} still satisfy certain space-time periodicity properties.
Indeed, there exists a family $(e'_\alpha)_{\alpha\in[n]}$ such that $e'_n=e$, $e'_\alpha\in\Q^n$ for
all $\alpha\in[n-1]$, and $(e'_\alpha)_{\alpha\in[n-1]}$ is an orthonormal basis of $e^\perp$ 
\cite[Lemma 4.2]{Berestycki_Hamel_Nadin}\revision{, associated with an orthogonal coordinate
system $(x'_\alpha)_{\alpha\in[n-1]}$}. Hence there exists a large integer $M\in\N$ such that 
$Me'_\alpha\in\Z^n$ for any $\alpha\in[n]$ and $cMe\in\Z$, and therefore, for any coefficient $f$ of 
\eqref{sys:KPP_moving_frame} and any family $(\epsilon_{\alpha-1})_{\alpha\in[n+1]}\in\{0,1\}^{n+1}$,
\begin{equation}\label{eq:periodicity_new_variables}
    \begin{split}
	f^{\tl}\left( t+\epsilon_0 M,x+M\sum_{\alpha\in[n]}\epsilon_{\alpha}e'_\alpha \right) & 
	= f\left( t+\epsilon_0 M,x-cte+\sum_{\alpha\in[n]}M\epsilon_{\alpha}e'_\alpha-cM\epsilon_0 e \right) \\
	& = f(t,x-cte) \\
	& = f^{\tl}(t,x)
    \end{split}
\end{equation}
This means that $f^{\tl}$ is $M$-periodic with respect to time and with respect to each new orthogonal spatial
coordinate $x'_\alpha$. 
From now on, we denote $T_{e,c}$ and $L_{e,\alpha}$ the respective minimal periods in the variables $t$ and 
$x'_\alpha$, $\alpha\in[n]$, for the whole (finite) family of coefficients (note that 
$L_e=(L_{e,\alpha})_{\alpha\in[n]}$ does not depend on $c$).

\begin{defi}\label{defi:PTW_transverse_space}
    Let $e\in\Q^n$ and $c\in\Q$.

    A classical entire solution $\veu$ of \eqref{sys:KPP} is a 
    \revision{\textit{planar pulsating traveling wave in direction $-e\in\Sn$ at speed $c\in\R$ with time and transverse 
    space periodicity}} if:
    \begin{enumerate}
	\item it is an almost planar generalized transition wave in direction $-e$ with global mean speed $c$;
	\item $\veu^{\tl}:(t,x)\mapsto\veu(t,x-cte)$ is $T_{e,c}$-periodic with respect to $t$ and, for each $\alpha\in[n-1]$,
    $L_{e,\alpha}$-periodic with respect to the spatial coordinate $x_\alpha'$ of $e^\perp$.
    \end{enumerate}
\end{defi}

\revision{For clarity, from now on we will refer to the solutions defined in Definition \ref{defi:PTW_whole_space} as
planar pulsating traveling wave in direction $-e\in\Sn$ at speed $c\in\R$ with time and \textit{whole} space periodicity.}

Theorems \ref{thm:supercritical_existence} and \ref{thm:critical_existence} below will confirm the existence of
\revision{planar pulsating traveling wave in direction $-e\in\Sn$ at speed $c\in\R$ with time and transverse space periodicity.}

\subsection{The linear part and the nonlinear part}

For brevity, we will denote from now on 
\begin{equation*}
    \cbQ=\dcbP-\veL
\end{equation*}
the linear operator derived from the linearization of \eqref{sys:KPP} at $\veu=\vez$. By virtue of
the assumptions \ref{ass:cooperative}--\ref{ass:irreducible}, this linear operator is cooperative and fully coupled, and 
this will be a key property in the forthcoming analysis of the non-cooperative semilinear system \eqref{sys:KPP}.
On the contrary, the nonlinear remainder of the reaction term $-(\veB\veu)\circ\veu$ plays a secondary role, and results can be generalized with 
minor technical adaptations to reaction terms of the form $\veL(t,x)\veu-\vect{f}\left(t,x,\veu\right)\circ\veu$
satisfying the following assumption of KPP type:
\begin{equation*}
    \forall(t,x)\in\R\times\R^n,\quad
    \begin{dcases}
	\vect{f}(t,x,\vez)=\vez, \\
	\vect{f}(t,x,\vev)\geq\vez\quad\text{if }\vev\geq\vez.
    \end{dcases}
\end{equation*}
Such a generalization can be found in a previous paper by the first author \cite{Girardin_2016_2}.
The motivation of the restriction in the present paper is twofold: on one hand, the form $\veL\veu-\veB\veu\circ\veu$ 
is sufficient for the applications we have in mind; on the other hand, it minimizes the verbosity and highlights
the role of the linear part.

\subsection{Notations}

Generally speaking, notations are chosen consistently with a previous paper on space-time
homogeneous coefficients by the first author \cite{Girardin_2016_2} and with a paper by I. Mazari and the first author
on the principal spectral analysis of $\cbQ$ \cite{Girardin_Mazari_2022}.

In the whole paper, $\N$ is the set of nonnegative integers, which contains $0$.

Unless specified otherwise, time and space periodicities refer to, respectively, $1$-periodicity with respect to $t$ and $1$-periodicity with respect to 
$x_\alpha$ for each $\alpha\in[n]$ (or $1$-periodicity with respect to $x$ for short). 

Vectors in $\R^N$ and matrices in $\R^{N\times N}$ are denoted in bold font. 
Functional operators are denoted in calligraphic typeface (bold if they act on functions valued in $\R^N$).
Functional spaces, \textit{e.g.} $\mathcal{W}^{1,\infty}(\R\times\R^n,\R^N)$, 
are also denoted in calligraphic typeface. A functional space $\mathcal{X}$ denoted with a subscript $\mathcal{X}_\upp$, 
$\mathcal{X}_{t-\upp}$ or $\mathcal{X}_{x-\upp}$ is restricted to functions that are space-time periodic, time periodic or space periodic respectively. Lebesgue and Sobolev spaces restricted to space-time, time or space periodic functions correspond naturally
to integrals over the space-time, time or space periodicity cell respectively. 

For clarity, H\"{o}lder spaces of functions with $k\in\mathbb{N}$ derivatives that are all
H\"{o}lder-continuous with exponent $\alpha\in(0,1)$ are denoted $\caC^{k+\alpha}$; when the domain is $\R\times\R^n$, it 
should be unambiguously understood that $\caC^{k+\alpha,k'+\alpha'}$ denotes the set of functions that have $k$ 
derivatives in time and $k'$ derivatives in space, which are all uniformly $\alpha$-H\"{o}lder-continuous in time and 
$\alpha'$-H\"{o}lder-continuous in space.

For any two vectors $\veu,\vev\in\R^N$, $\veu\leq\vev$ means $u_i\leq v_i$ for all 
$i\in[N]$, $\veu<\vev$ means $\veu\leq\vev$ together with $\veu\neq\vev$ and $\veu\ll\vev$ means $u_i<v_i$ for all $i\in[N]$. If 
$\veu\geq\vez$, we refer to $\veu$ as \textit{nonnegative}; if $\veu>\vez$, as \textit{nonnegative nonzero}; if $\veu\gg\vez$, as 
\textit{positive}. The sets of all nonnegative, nonnegative nonzero, positive vectors are respectively denoted $[\vez,\vei)$,
$[\vez,\vei)\backslash\{\vez\}$ and $(\vez,\vei)$. The vector whose entries are all equal to $1$ is denoted $\veo$ and this never refers to
an indicator function.
Similar notations and terminologies might be used in other dimensions and for matrices. The identity matrix is denoted $\vect{I}$.

Similarly, a function can be nonnegative, nonnegative nonzero, positive. For clarity, a positive function is a function with only positive values.

To avoid confusion between operations in the state space $\R^N$ and operations in the spatial domain $\R^n$, 
Latin indexes $i,j,k$ are assigned to vectors and matrices of size $N$ whereas Greek indexes $\alpha,\beta,\gamma$ 
are assigned to vectors and matrices of size $n$. 
We use mostly subscripts to avoid confusion with algebraic powers, but when both Latin and Greek indexes are involved, we 
move the Latin ones to a superscript position, \textit{e.g.} $A^i_{\alpha,\beta}(t,x)$.
We denote scalar products in $\R^N$ with the transpose operator, $\veu^\upT\vev=\sum_{i=1}^N u_i v_i$,
and scalar products in $\R^n$ with a dot, $x\cdot y =\sum_{\alpha=1}^n x_\alpha y_\alpha$. 

For any vector $\veu\in\R^N$, $\diag(\veu)$, $\diag(u_i)_{i\in[N]}$ or $\diag(u_i)$ for short refer to the diagonal matrix in $\R^{N\times N}$
whose $i$-th diagonal entry is $u_i$. These notations can also be used if $\veu$ is a function valued in $\R^N$.

Finite dimensional Euclidean norms are denoted $|\cdot |$ whereas the notation $\|\cdot \|$ is reserved for norms in functional spaces.

The notation $\circ$ is reserved in the paper for the Hadamard product (component-wise product of vectors or matrices)
and never refers to the composition of functions.
For any vector $\vev\in\R^N$ and $p\in\R$, $\vev^{\circ p}$ denotes the vector $(v_i^p)_{i\in[N]}$.

\subsection{Main result}

Before stating the main result, we need to introduce the periodic principal eigenvalue of $\cbQ$:
\begin{equation}
    \lambda_{1,\upp}=\inf\left\{ \lambda\in\R\ |\ \exists \veu\in\mathcal{W}^{1,\infty}\cap\caC^{1,2}_{t-\upp}(\R\times\R^n,(\vez,\vei))\ \cbQ\veu\leq\lambda\veu \right\},
\end{equation}
and to recall results from the work of the first author on persistence, extinction and spreading properties of
the solutions of the Cauchy problem associated with \eqref{sys:KPP} \cite{Girardin_2023}.

By virtue of \cite[Theorem 1.1]{Girardin_2023}, if $\lambda_{1,\upp}\geq 0$, all solutions of
the Cauchy problem go extinct uniformly in space and, in particular, there exists no pulsating
traveling wave solution.
On the contrary, by \cite[Theorem 1.3, Theorem 1.4]{Girardin_2023}, if $\lambda_{1,\upp}<0$, 
solutions of the Cauchy problem whose initial values have a sufficiently slow exponential decay 
persist locally uniformly and spread in space, in the following sense:
\begin{equation}
    \liminf_{t\to+\infty}\min_{i\in[N]}\inf_{|x|\leq R}u_i(t,x)>0\quad\text{for all }R>0.
    \label{eq:Cauchy_persistence}
\end{equation}
These solutions are natural candidates for long-time convergence (in some sense) to pulsating 
traveling waves.

To state our main result, we also need to define a modified operator parameterized by a
pair $(\mu,e)\in[0,+\infty)\times\Sn$:
\begin{equation}
    \cbQ_{\mu e}=\cbQ - \diag\left( 2\mu A_i e\cdot\nabla \right) - \diag\left(\mu^2 e\cdot A_i e+\mu\nabla\cdot(A_i e)-\mu q_i\cdot e \right)
    \label{eq:Qmue}
\end{equation}
and its periodic principal eigenvalue $\lambda_{1,\mu e}=\lambda_{1,\upp}(\cbQ_{\mu e})$.
Defining $\upe_{\mu e}:x\mapsto\upe^{\mu e\cdot x}$, it can be verified that the periodic 
principal eigenfunction $\veu_{\mu e}$ of $\cbQ_{\mu e}$, normalized by 
$\max_{i\in[N]}\max_{(t,x)\in\R\times\R^n}u_{\mu e,i}(t,x)=1$, satisfies:
\begin{equation}
    \cbQ(\upe_{\mu e}\veu_{\mu e})=\lambda_{1,\mu e}\upe_{\mu e}\veu_{\mu e}.
    \label{eq:ppe_Qmue}
\end{equation}
We also define
\begin{equation}
    c_e^\star=\min_{\mu>0}\frac{-\lambda_{1,\mu e}}{\mu}.
    \label{eq:minimal_wave_speed}
\end{equation}
The fact that the function $\mu\mapsto\frac{-\lambda_{1,\mu e}}{\mu}$ admits a unique global minimum is 
standard in KPP-type problems and will be recalled below in Lemma \ref{lem:minimal_c}. Note that this minimum
could be negative. 

Our main results are the following ones.

\begin{thm}\label{thm:nonexistence}
    Assume $\lambda_{1,\upp}<0$. 

    Then, for any $e\in\Sn$ and $c\in(-\infty,c_e^\star)$, there exists no almost planar generalized transition wave solution in
    direction $-e$ with global mean speed $c$.
\end{thm}

\begin{thm}\label{thm:supercritical_existence}
    Assume $\lambda_{1,\upp}<0$. 

    Then, for any $e\in\Sn\cap\Q^n$ and $c\in(c_e^\star,+\infty)\cap\Q$, there exists a \revision{planar pulsating traveling wave 
    solution in direction $-e$ with speed $c$ with time and transverse space periodicity}.
\end{thm}

\begin{thm}\label{thm:critical_existence}
    Assume $\lambda_{1,\upp}<0$. 

    Then, for any $e\in\Sn\cap\Q^n$ such that $c_e^\star\in\Q$, there exists a \revision{planar pulsating traveling wave solution 
    in direction $-e$ with speed $c_e^\star$ with time and transverse space periodicity}.
\end{thm}

In the special case where coefficients are spatially homogeneous, the same method of proof of existence works with no rationality assumption
and gives the existence of almost planar generalized transition waves with profiles that are 
$1$-periodic with respect to $t$ and homogeneous with respect to the $n-1$ coordinates $x_\alpha'$ of $e^\perp$. 
Usually, when coefficients are independent of space and periodic in time, such solutions form the definition of 
\revision{planar pulsating traveling waves with time periodicity} \cite{FrejacquesPhD,Alikakos_Bates_Chen_1999}.
This corresponds to a particular case of Definition \ref{defi:PTW_whole_space}.

\begin{cor}\label{cor:space_homogeneous}
    Assume $\lambda_{1,\upp}<0$. 

    Assume moreover that the coefficients of \eqref{sys:KPP} are homogeneous with respect to space.

    Then, for any $e\in\Sn$ and $c\in\R$, there exists a \revision{planar pulsating traveling wave solution
    in direction $-e$ with speed $c$ with time periodicity} if and only if $c\geq c_e^\star$.
\end{cor}

\subsection{Comments}

\revision{It might be tempting to use the density of $\Q$ in $\R$ together with Theorems 
\ref{thm:supercritical_existence} and \ref{thm:critical_existence} to construct an entire solution
for any direction $-e$ and any speed $c$. However, the sequences of space periods and time periods
associated with the sequence of \revision{planar pulsating traveling waves with time and transverse space periodicity}
blow up when passing to the limit. Thus such a limiting procedure constructs, at best, almost planar
generalized transition waves.} Nonetheless, the following conjecture is natural.

\begin{conj}\label{conj:existence_of_SPPTW}
    For all $e\in\Sn$ and $c\geq c_e^\star$, there exists a \revision{planar pulsating traveling wave solution in direction $-e$
    with speed $c$ with time and whole space periodicity}.
\end{conj}

Conjecture \ref{conj:existence_of_SPPTW} is known to be true in particular cases mentioned above 
(cf. Section \ref{sec:PPTW} and Corollary \ref{cor:space_homogeneous}).
Nevertheless, at the highest degree of generality, its proof remains entirely open.
It is indeed not clear that the \revision{planar pulsating traveling waves with time and transverse space periodicity}
provided by Theorems \ref{thm:supercritical_existence} and \ref{thm:critical_existence} are actually 
\revision{planar pulsating traveling waves with time and whole space periodicity}.

In \cite{Girardin_2023}, where the Cauchy problem associated with \eqref{sys:KPP} was studied, 
the construction of pulsating traveling waves was mentioned as an open problem. 
Our results bring partial answers. The proof of Conjecture \ref{conj:existence_of_SPPTW} would
complete the answer for planar waves.

\revision{The formula \eqref{eq:minimal_wave_speed} is related with the so-called Freidlin--G\"{a}rtner formula 
\cite{Gartner_Freidlin,Freidlin_1985}, which was originally introduced as the asymptotic spreading speed of solutions of the 
Cauchy problem for the scalar equation $N=1$ with space periodic coefficients and with compactly supported initial conditions.
In one-dimensional space $n=1$, this asymptotic spreading speed and the minimal planar wave speed $c_e^\star$ coincide, but
in higher dimensions they might differ in a way that is precisely quantified by the Freidlin--G\"{a}rtner formula. The
Freidlin--G\"{a}rtner formula was confirmed to hold for \eqref{sys:KPP} in \cite{Girardin_2023}.}

The fact that $c_e^\star$ might be of any sign is in particular due to the presence
of arbitrary advection terms $(q_i)_{i\in[N]}$. For instance, recalling that the one-dimensional 
scalar KPP equation $\partial_t u = \partial_{xx}u + u(1-u)$ admits $2$ as bidirectional minimal 
wave speed, it follows that the shifted equation 
$\partial_t u = \partial_{xx}u+q\partial_x u + u(1-u)$, with any $q\in\R$, admits $2\pm q$ as 
leftward and rightward minimal wave speeds.

Circumventing the absence of existence--comparison principle for the semilinear non-cooperative
system will be achieved with a fixed point argument similar to
the one in \cite{Girardin_2016_2}, itself inspired by \cite{Berestycki_Nadin_Perthame_Ryzhik}. 
In this sense, this technical circumvention is not a novelty of this paper, even though it is the first
time to the best of our knowledge it is applied in a non-cooperative space-time periodic framework.

It follows from \cite[Corollary 2.2]{Girardin_2023} that all almost planar generalized transition waves
are globally bounded from above by the same constant vector $K\veo$, with a constant $K>0$ 
that only depends on bounds on $\veL$ and $\veB$. From this uniform boundedness, the uniform positivity
in the upstream asymptotic in Definition \ref{defi:GTW}
will follow from an argument involving the Harnack inequality, as in \cite{Girardin_2016_2}.

Apart from its uniform positivity and boundedness, the behavior of the profiles $\veu^{\tl}(t,x+ze)$ as 
$z\to+\infty$ is unknown when $N>1$. We recall that, contrarily to the scalar case, there might be 
multiple uniformly positive entire solutions of \eqref{sys:KPP} susceptible of being the 
limit as $z\to+\infty$. In fact it is not even clear that the profiles converge as 
$z\to+\infty$. This observation is neither new nor specific to the space-time periodic setting
\cite{Girardin_2016_2,Girardin_2017}.  

The behavior of the profiles $\veu^{\tl}(t,x+ze)$ as $z\to-\infty$ is less elusive.
As in the scalar case \cite{Nadin_2009} or in the case of systems with homogeneous coefficients
\cite{Girardin_2016_2}, exponential super- and sub-estimates follow from the construction. However
these apply \textit{a priori} only to the constructed profiles, not to all existing profiles.
The exponential approximation of all profiles requires much more work, \textit{cf.} \cite{Girardin_2017}
in the case of homogeneous coefficients or \cite{Hamel_2008} in the case of scalar equations
with space periodic coefficients. This is outside the scope of this paper.

We emphasize that the minimal wave speed in direction $-e$ is denoted $c^\star_e$ and
not $c^\star_{-e}$. This change in notation compared to the previous work of the first author \cite{Girardin_2023} 
is done for notational convenience and for consistency with the work of the second author on the scalar case
\cite{Nadin_2009}. It should not be a cause for confusion.

\revision{Finally, we mention the article \cite{Huang_Wu_Zhao_2025}, where a method similar to ours
    is used to construct what is called here almost planar generalized transition waves, and there
    transition semi-waves, for a related class of one-dimensional two-species non-cooperative systems 
    with space-time periodic coefficients, possible time delay, and no diffusion for the second species.
    It is a rare example of work on the spreading properties of solutions to non-cooperative systems with
    space-time periodic coefficients. Their article motivates a possible sequel to our article, on systems 
with degenerate diffusion and time delays.}

\section{Preliminaries}
In this section, we establish or recall basic results that will be used repeatedly in the main proofs.

\subsection{Harnack inequality for linear cooperative systems}

For self-containment and ease of reading, we recall here the Harnack inequality 
that was proved in \cite{Girardin_Mazari_2022}. It is a 
refinement of \FPH inequality \cite[Theorem 3.9]{Foldes_Polacik_2009}
for parabolic cooperative systems and Arapostathis--Ghosh--Marcus's Harnack inequality 
\cite[Theorem 2.2]{Araposthathis_} for elliptic cooperative systems.

We denote by $\sigma>0$ the smallest positive entry of $\overline{\veL}$ (\textit{cf.} \ref{ass:irreducible})
and by $K\geq 1$ the smallest positive number such that
\[
    K^{-1}\leq\min_{i\in[N]}\min_{y\in\Sn}\min_{(t,x)\in\clOmper}\left(y\cdot A_i(t,x)y\right),
\]
\[
    \max_{i\in[N]}\max_{y\in\Sn}\max_{(t,x)\in\clOmper}\left(y\cdot A_i(t,x)y\right)\leq K,
\]
\[
    \max_{i\in[N]}\max_{\alpha\in[n]}\max_{(t,x)\in\clOmper}|q_{i,\alpha}(t,x)|\leq K,
\]
\[
    \max_{i,j\in[N]}\sup_{(t,x)\in\clOmper}|l_{i,j}(t,x)|\leq K.
\]
The existence of $K$ is given by \ref{ass:ellipticity} and \ref{ass:smooth_periodic}.

\begin{prop}[Proposition 2.4 in \cite{Girardin_Mazari_2022}]\label{prop:harnack_inequality}
Let $\theta\geq\max\left(T,L_1,\dots,L_n\right)$ and $\vect{f}\in\mathcal{L}^\infty\cap\caC^{\delta/2,\delta}(\R\times\R^n,\R^N)$, 
where $\delta\in(0,1)$ is as in \ref{ass:smooth_periodic}. Let $F>0$ such that  
\[
    \max_{i\in[N]}\sup_{(t,x)\in\R\times\R^n}|f_i(t,x)|\leq F.
\]

There exists a constant $\overline{\kappa}_{\theta,F}>0$, determined only by $n$, $N$, $\sigma$, $K$ and the parameters
$\theta$ and $F$ such that, if $\veu\in\caC([-2\theta,6\theta]\times[-\frac{3\theta}{2},\frac{3\theta}{2}]^n,[\vez,\vei))$ 
is a solution of $\cbQ\veu=\diag(\vect{f})\veu$, then
\[
    \min_{i\in[N]}\min_{(t,x)\in[5\theta,6\theta]\times[-\frac{\theta}{2},\frac{\theta}{2}]^n}u_i(t,x)
    \geq \overline{\kappa}_{\theta,F}\max_{i\in[N]}\max_{(t,x)\in[0,2\theta]\times[-\frac{\theta}{2},\frac{\theta}{2}]^n}u_i(t,x).
\]
\end{prop}

\subsection{Existence of the minimal wave speed}

\begin{lem}[Lemma 3.5 in \cite{Girardin_2023}]\label{lem:minimal_c}
    Let $e\in\Sn$. Assume $\lambda_{1,\upp}<0$.

    The infimum $c_e^\star\in\R$ of the image of the function $\mu\in[0,+\infty)\mapsto\frac{-\lambda_{1,\mu e}}{\mu}$
    is a minimum attained at a unique $\mu^\star_e>0$.

    Moreover, for any $c>c^\star_e$, there exists $\muwed_e,\muvee_e>0$ such that $\muwed_e<\mu^\star_e<\muvee_e$ and
    $\frac{-\lambda_{1,\muwed_e e}}{\muwed_e}=\frac{-\lambda_{1,\muvee_e e}}{\muvee_e}=c$.
\end{lem}

\section{Proof of Theorem \ref{thm:nonexistence}}

\revision{This proof uses the following results from the previous work by the first author. Their proof is omitted here.

\begin{lem}[Corollary 2.2 in \cite{Girardin_2023}]\label{lem:global_bound_entire}
    There exists a constant $K>0$ such that all nonnegative globally bounded entire solutions $\veu$ of \eqref{sys:KPP} satisfy
    \begin{equation*}
	\veu\leq K\veo\quad\text{in }\R\times\R^n.
    \end{equation*}
\end{lem}

\begin{lem}[Corollary 2.7 in \cite{Girardin_2023}]\label{lem:comparison_of_components_in_real_time}
    Let $M>0$ and $\veu_{\upini}\in\caC_b(\R^n,[0,M]^N)$.

    Then there exists $p\in(0,1)$ and $\kappa_M>0$ such that the solution $\veu$ of the Cauchy problem
    associated with \eqref{sys:KPP} and with the initial condition
    \begin{equation}
	\veu(0,x)=\veu_{\upini}(x)\quad\text{for all }x\in\R^n.
    \end{equation}
    satisfies
    \begin{equation*}
	\forall i,j\in[N]\quad u_j(t,x)\leq \kappa_M u_i(t,x)^p\quad\text{for all }(t,x)\in[1,+\infty)\times\R^n.
    \end{equation*}
\end{lem}}

\begin{proof}[Proof of Theorem \ref{thm:nonexistence}]
    Assume there exists $e\in\Sn$, $c<c_e^\star$ and an almost planar generalized transition wave $\veu$
    in direction $-e$ with global mean speed $c$.
    Note that the limit \ref{eq:definition_GTW_limit_negative_infinity} in Definition \ref{defi:GTW} implies
    \begin{equation}
	\veu\left(t,-\frac{c+c_e^\star}{2}te\right)=\veu^{\tl}\left( t,-\frac{c_e^\star-c}{2}te \right)\to \vez
	\quad\text{as }t\to+\infty.
	\label{eq:nonexistence_slow_spreading}
    \end{equation}

    \revision{By combining Lemmas \ref{lem:global_bound_entire} and \ref{lem:comparison_of_components_in_real_time},}
    it \revision{follows} that, for a certain $D>0$ and $p\in(0,1)$,
    \begin{equation*}
	\veL\veu - \veB\veu\circ\veu\geq\veL\veu-D\veu^{\circ(1+p)}\quad\text{in }\R\times\R^n.
    \end{equation*}
    This is a key comparison between the non-cooperative reaction term and a semilinear cooperative reaction term.
    Hence, by virtue of the comparison principle applied to the cooperative operator 
    $\vev\mapsto\cbQ\vev+D\vev^{\circ(1+p)}$, for any initial condition $\underline{\veu}_0$ satisfying, for
    all $x\in\R^n$, $\underline{\veu}_0(x)\leq\veu(0,x)$, the solution 
    $\underline{\veu}:[0,+\infty)\times\R^n\to\R^N$ of the Cauchy problem
    \begin{equation*}
	\begin{dcases}
	    \cbQ\underline{\veu} = -D\underline{\veu}^{\circ(1+p)} & \text{in }(0,+\infty)\times\R^n, \\
	    \underline{\veu}(0,\cdot)=\underline{\veu}_0 & \text{in }\R^n,
	\end{dcases}
    \end{equation*}
    satisfies $\veu\geq\underline{\veu}$ in $[0,+\infty)\times\R^n$.
    Hence, from \eqref{eq:nonexistence_slow_spreading}, 
    \begin{equation}
	\underline{\veu}\left(t,-\frac{c+c_e^\star}{2}te\right)\to \vez\quad\text{as }t\to+\infty.
	\label{eq:nonexistence_slow_spreading_subsolution}
    \end{equation}

    Now we specify $\underline{\veu}_0:x\mapsto \varepsilon H(x\cdot e)\veo$, where $H$ is the Heavyside function
    equal to $0$ in $(-\infty,0)$ and to $+1$ in $(0,+\infty)$, and where $\varepsilon$ is chosen appropriately small
    so that $\underline{\veu}_0(\cdot)\leq\veu(0,\cdot)$ holds true (this is indeed possible because
    $\veu$ is pointwise positive and, as $x\cdot e\to+\infty$, it satisfies the limit \ref{eq:definition_GTW_limit_positive_infinity}).

    Since $\vev\mapsto\cbQ\vev+D\vev^{\circ(1+p)}$ is a cooperative operator with space-time periodic
    coefficients, we can apply to $\underline{\veu}$ a known spreading result for front-like
    initial data \cite[Theorem 2.1]{Du_Li_Shen_2022} and deduce that $\underline{\veu}$ spreads at speed 
    $c_e^\star$, and in particular
    \begin{equation*}
	\liminf_{t\to+\infty}\min_{i\in[N]}\underline{u}_i\left(t,-\frac{c+c_e^\star}{2}te\right)>0.
    \end{equation*}
    But this contradicts directly \eqref{eq:nonexistence_slow_spreading_subsolution}.
\end{proof}

\section{Proof of Theorem \ref{thm:supercritical_existence}}

In this section, we assume $\lambda_{1,\upp}<0$ and we prove the existence of pulsating traveling 
waves for all $e\in\Sn\cap\Q^n$ and $c\in(c^\star_e,+\infty)\cap\Q$.

The wave direction $e$ and the wave speed $c$ are fixed once and for all. 
Recall the notations $T_{e,c}$ and $L_e$ for the minimal periods of the coefficients of \eqref{sys:KPP_moving_frame}
\revision{deduced from \eqref{eq:periodicity_new_variables}}. 
For brevity, we omit thereafter the subscripts $e$ and $c$ in all notations introduced earlier.
\revision{In this whole section, time periodicity implicitly refers to $T$-periodicity and space periodicity
implicitly refers to $L$-periodicity.}

\subsection{\revision{Recasting the problem by changing the variables and the coefficients}}\label{sec:change_of_variables_and_coefficients}

Let $P\in\R^{n\times n}$ be the orthogonal matrix associated with the change of variables replacing $x$ \revision{(the coordinates
in the canonical basis of $\R^n$)} into $x'$ \revision{(the coordinates in the new orthonormal basis $(e_\alpha)_{\alpha\in[n]}$)}, \textit{i.e.}
\begin{equation}
    \begin{pmatrix}x_1 \\ \vdots \\ x_{n-1} \\ x_n\end{pmatrix}
    = P\begin{pmatrix}x'_1 \\ \vdots \\ x'_{n-1} \\ x'_n\end{pmatrix}.
    \label{eq:change_of_variables}
\end{equation}
Remarking that $\nabla_{x'} = P^{\upT}\nabla_x = P^{-1}\nabla_x$, we set
\begin{equation*}
    A_i':(t,x')\mapsto P^{\upT}A_i^{\tl}(t,Px')P,
\end{equation*}
\begin{equation*}
    q_i':(t,x')\mapsto P^{\upT}(q_i^{\tl}+ce)(t,Px'),
\end{equation*}
\begin{equation*}
    \veL':(t,x')\mapsto\veL^{\tl}(t,Px'),
\end{equation*}
\begin{equation*}
    \veB':(t,x')\mapsto\veB^{\tl}(t,Px'),
\end{equation*}
\begin{equation}
    \cbR = \partial_t-\diag\left(\nabla_{x'}\cdot\left(A_i'\nabla_{x'}\right)\right)+\diag\left(q_i'\cdot \nabla_{x'}\right) - \veL'.
    \label{defi:operator_R}
\end{equation}
By construction, if $\veu'$ is a solution of $\cbR\veu'=-\veB'\veu'\circ\veu'$, 
then $\veu:(t,x)\mapsto\veu'(t,P^{\upT}(x+cte))$ is a solution of $\cbQ\veu=-\veB\veu\circ\veu$.

In the following construction, we only work with the variables $(t,x')$, never with the variables $(t,x)$. 
In order to simplify notations, we will prefer the generic notation $x'=(y,z)\in\R^{n-1}\times\R$, 
\textit{i.e.} $x_\alpha'=y_\alpha$ if $\alpha\in[n-1]$ and $x_n'=z$. 
The spatial periodicity cell $(0,L)=(0,L_1)\times(0,L_2)\times\dots\times(0,L_n)$ is then denoted $(0,L_y)\times(0,L_z)$, with
$L_y\in\R^{n-1}$ and $L_z\in\R$.
The notation $\nabla$ should be unambiguously understood as $\nabla_{x'}=\nabla_{(y,z)}$ 
in this construction. 
Also, in this coordinate system, $e$ becomes $e'=P^{\upT}e=(0,0,\dots,0,1)^{\upT}$,
$e\cdot\nabla_x$ becomes $e'\cdot\nabla=\partial_{x'_n}=\partial_z$ and $\nabla_x\cdot(A_i^{\tl} e)$ becomes 
$\nabla\cdot(A_i'e')$.

The family $(A_i')_{i\in[N]}$ of diffusivity matrices in the new coordinate system remains uniformly
elliptic, \textit{i.e.} it still satisfies \ref{ass:ellipticity}.
For any $\mu\geq 0$, the operator $\cbQ_{\mu}$ defined in \eqref{eq:Qmue} becomes 
\begin{equation*}
    \cbR_\mu = \cbR - \diag\left( 2\mu A_i'e'\cdot\nabla \right) - \diag\left(\mu^2 e'\cdot A_i' e'+\mu\nabla\cdot(A_i' e')-\mu q_i'\cdot e' \right).
\end{equation*}

\revision{The main interest of these changes of variables and coefficients is outlined in the following key proposition.
Recall that space and time periodicities are implicitly $L$- and $T$-periodicities.

\begin{prop}\label{prop:transformed_WPPTW_problem}
    The coefficients of the operator $\cbR$ are space-time periodic. 
    So are the coefficients of $\cbR_\mu$ for any $\mu\geq 0$.

    The set of \revision{planar pulsating traveling wave solutions in direction $-e$ at speed $c$ with time and transverse
    space periodicity} is in bijection with the set
    of nonnegative globally bounded classical solutions $\veu'\in\caC^{1,2}_{\upp}(\R\times\R^{n-1},\caC^2(\R,\R^N))$ 
    of the following problem:
    \begin{equation}\label{sys:WPPTW_new_coordinates_new_coefficients}
	\begin{dcases}
	    \cbR\veu' = -(\veB'\veu')\circ\veu' & \text{in }\R\times\R^{n-1}\times\R, \\
	    \lim_{z\to-\infty}\max_{(t,y)\in[0,T]\times[0,L_y]}\max_{i\in[N]}u'_i(t,y,z)=0, \\
	    \liminf_{z\to+\infty}\min_{(t,y)\in[0,T]\times[0,L_y]}\min_{i\in[N]}u'_i(t,y,z)>0.
	\end{dcases}
    \end{equation}
\end{prop}

\begin{proof}
    We refer to \eqref{eq:periodicity_new_variables}, to Definition \ref{defi:PTW_transverse_space} and to the above construction.
\end{proof}}

Also, by construction, $\veu_\mu':(t,x')\mapsto\veu_\mu(t,Px'-cte)$ is a space-time periodic principal eigenfunction
of $\cbR_\mu$ (with periods $L$ and $T$).
Applying the Krein--Rutman theorem, we deduce by uniqueness that all space-time periodic principal eigenfunctions of
$\cbR_\mu$ with the same periods are positively proportional to $\veu_\mu'$. 

Finally, $(t,x)\mapsto\upe^{\mu e\cdot x}\veu_\mu(t,x)$ becomes $(t,y,z)\mapsto\upe^{\mu ct}\upe^{\mu z}\veu_\mu'(t,y,z)$,
and the identity \eqref{eq:ppe_Qmue} becomes:
\begin{equation}
    \cbR(\upe_{\mu e'}\veu_\mu')=(\lambda_{1,\mu}+c\mu)\upe_{\mu e'}\veu_\mu'.
    \label{eq:principal_eigenfunction_new_coordinates}
\end{equation}

For clarity, we divide the proof into parts.

\subsection{Systems in truncated cylinders}

In this subsection we fix $a>0$. We define the truncated cylinder
$\Sigma_a=(0,T)\times(0,L_y)\times(-a,a)$. We are going to construct a solution $\veu_a$ of the semilinear system 
$\cbR \veu = -\veB'\veu\circ\veu$ in $\Sigma_a$, 
with periodic boundary conditions in $(t,y)$ and appropriate inhomogeneous Dirichlet boundary conditions in $z$.

Although the semilinear non-cooperative system does not satisfy the comparison principle,
the construction, inspired by the traveling wave construction in homogeneous media \cite{Girardin_2016_2},
relies upon super- and sub-solutions, making it possible to apply a fixed point argument. 

Recall that, by virtue of the space-time
periodicity and of the Harnack inequality of Proposition \ref{prop:harnack_inequality} (see also \cite[Remark 2.1]{Girardin_Mazari_2022}) applied, \textit{mutatis mutandi}, to 
the operator $\cbR_{\mu}$, there exists $\kappa_\mu>0$ such that $\veu_\mu'\geq\kappa_\mu\veo$ 
globally in $\R\times\R^n$.
Moreover, it follows from the proof of the Harnack inequality that the Harnack constant can
be chosen in such a way that $\mu\in\R\mapsto\kappa_\mu$ is continuous.

\begin{lem}\label{lem:supersolution_in_truncated_domain}
    Let 
    \begin{equation}
	\overline{\veu}:(t,y,z)\mapsto\upe^{\muwed z}\veu_{\muwed}'(t,y,z).
	\label{defi:supersolution}
    \end{equation}

    Then it satisfies:
    \begin{equation}
	\cbR\overline{\veu}\geq\vez\quad\text{in }\Sigma_a.
	\label{eq:supersolution_in_truncated_domain}
    \end{equation}
\end{lem}

\begin{proof}
    The function $\overline{\veu}$ actually satisfies in $\Sigma_a$:
    \begin{equation*}
	\cbR\overline{\veu} = (\lambda_{1,\muwed}+c\muwed)\overline{\veu} = \vez.
    \end{equation*}
\end{proof}

\begin{lem}\label{lem:subsolution_in_truncated_domain}
    Let
    \begin{equation}
	\gamma=\frac{1}{2}\min\left\{\muwed,\muvee-\muwed\right\},
	\label{eq:def_gamma_subsolution}
    \end{equation}
    \begin{equation}
	M=\max\left\{\frac{1}{\kappa_{\muwed+\gamma}},\frac{N\displaystyle\max_{(i,j)\in[N]^2}\max_{(t,x)\in\clOmper}b_{i,j}(t,x)}{\left(\left( \lambda_{1,(\muwed+\gamma)e}+c(\muwed+\gamma) \right)\kappa_{\muwed+\gamma}\right)}\right\},
	\label{eq:def_M_subsolution}
    \end{equation}
    and
    \begin{equation}
	\underline{\veu}:(t,y,z)\mapsto\upe^{\muwed z}\veu_{\muwed}'(t,y,z)
	-M\upe^{(\muwed+\gamma)z}\veu_{\muwed+\gamma}'(t,y,z).
	\label{defi:subsolution}
    \end{equation}

    Then it satisfies:
    \begin{equation}
	\underline{\veu}\leq\vez\quad\text{in }\Sigma_a\cap\left\{ z\geq 0 \right\},
	\label{eq:subsolution_negative_in_R_plus}
    \end{equation}
    \begin{equation}
	\left( \cbR+\diag\left(\veB'\overline{\veu}\right)\right)\underline{\veu}\leq\vez\quad\text{in }\Sigma_a.
	\label{eq:subsolution_in_truncated_domain}
    \end{equation}
\end{lem}

\begin{proof}
    We fix an arbitrary $(t,y,z)\in\Sigma_a$ and we omit arguments in $(t,y)$ 
    for readability.

    On one hand, with the same calculations as in the proof of Lemma 
    \ref{lem:supersolution_in_truncated_domain}, 
    \begin{equation*}
	\cbR\underline{\veu}=-M\left(\lambda_{1,\muwed+\gamma}+c(\muwed+\gamma)\right)\upe^{(\muwed+\gamma)z}\veu_{\muwed+\gamma}'.
    \end{equation*}
    By definition of $\gamma$, $\muwed+\gamma\in(\muwed,\muvee)$, whence, by virtue of 
    Lemma \ref{lem:minimal_c},
    \begin{equation*}
	\lambda_{1,\muwed+\gamma}+c(\muwed+\gamma)>0.
    \end{equation*}
    Using $\veu_{\muwed+\gamma}'\geq\kappa_{\muwed+\gamma}\veo$ and then the definition of $M$,
    \begin{equation*}
	\begin{split}
	    -\cbR\underline{\veu} & =M\left(\lambda_{1,\muwed+\gamma}+c(\muwed+\gamma)\right)\upe^{(\muwed+\gamma)z}\veu_{\muwed+\gamma}' \\
	    & \geq M\left( \lambda_{1,\muwed+\gamma}+c(\muwed+\gamma) \right)\upe^{\left( \muwed+\gamma \right)z}\kappa_{\muwed+\gamma}\veo \\
	    & \geq \upe^{\left( \muwed+\gamma \right)z}N\max_{(i,j)\in[N]^2}\max_{\clOmper}b_{i,j}\veo.
	\end{split}
    \end{equation*}

    On the other hand, by definition of $\underline{\veu}$ and using $\veu_{\muwed}'\leq\veo$
    and $\veu_{\muwed+\gamma}'\geq\kappa_{\muwed+\gamma}\veo$,
    \begin{equation*}
	\underline{\veu}\leq\upe^{\muwed z}(1-M\upe^{\gamma z}\kappa_{\muwed+\gamma})\veo\leq\upe^{\muwed z}(1-M\upe^{\gamma z}\kappa_{\muwed+\gamma})^{+}\veo,
    \end{equation*}
    whence, by nonnegativity of $\veB'\overline{\veu}$, definition of $\overline{\veu}$ and using again
    $\veu_{\muwed}\leq\veo$,
    \begin{equation*}
	(\veB'\overline{\veu})\circ\underline{\veu} \leq \upe^{2\muwed z}(1-M\upe^{\gamma z}\kappa_{\muwed+\gamma})^{+}N\max_{(i,j)\in[N]^2}\max_{\clOmper}b_{i,j}\veo.
    \end{equation*}

    Note that $(1-M\upe^{\gamma z}\kappa_{\muwed+\gamma})^{+}> 0$ if and only if
    $z< z_0=-\frac{1}{\gamma}\ln\left( M\kappa_{\muwed+\gamma} \right)$.
    By definition of $M$, $z_0\leq 0$. By definition of $\gamma$, $\muwed>\gamma$.
    Hence we deduce the inequality 
    $\upe^{\left( \muwed+\gamma \right)z}\geq\upe^{2\muwed z}(1-M\upe^{\gamma z}\kappa_{\muwed+\gamma})^{+}$.
    
    Consequently $-\cbR\underline{\veu}\geq(\veB'\overline{\veu})\circ\underline{\veu}$ in
    $\Sigma_a$ and the proof is ended.
\end{proof}

The function $\vez$ is obviously another sub-solution for the linear cooperative operator 
$\cbR+\diag(\veB'\overline{\veu})$.
Therefore the component-by-component maximum of $\underline{\veu}$ and $\vez$, denoted hereafter
$\underline{\veu}\vee\vez$, still provides a sub-solution, that can be used in a generalized
version of the existence--comparison principle for linear cooperative operators. 
This is stated in the following corollary.

\begin{cor}\label{cor:generalized_subsolution_system_truncated_domain}
    With the notations of Lemma \ref{lem:subsolution_in_truncated_domain}, 
    \begin{equation}
	\left( \cbR+\diag\left(\veB'\overline{\veu}\right)\right)\left(\underline{\veu}\vee\vez\right)\leq\vez\quad\text{in }\Sigma_a
	\label{eq:generalized_subsolution_system_truncated_domain}
    \end{equation}
    where partial derivatives are understood in distributional sense where necessary. The function 
    $\underline{\veu}\vee\vez$ is globally Lipschitz-continuous.
\end{cor}

We are now in a position to construct a bounded nonnegative entire solution of $\cbR\veu=-\veB'\veu\circ\veu$ in $\Sigma$ with 
appropriate boundary conditions.

\begin{prop}\label{prop:existence_profile_in_truncated_domain}
    Let $a^\star>0$ such that
    \begin{equation}
	(\underline{\veu}\vee\vez)(t,y,-a^\star)=\underline{\veu}(t,y,-a^\star)\gg\vez\quad\text{for all }(t,y)\in[0,T]\times[0,L_y].
	\label{eq:values_of_a_large_enough}
    \end{equation}

    Then, for all $a\geq a^\star$, there exists a solution $\veu$ of
    \begin{equation}
	\begin{dcases}
	    \cbR\veu = -(\veB'\veu)\circ\veu & \text{in }\Sigma_a, \\
	    \veu(t,y,\pm a)=(\underline{\veu}\vee\vez)(t,y,\pm a) & \text{for all }(t,y)\in[0,T]\times[0,L_y], \\
	    (t,y)\mapsto[z\mapsto\veu(t,y,z)] & \text{is periodic}.
	\end{dcases}
	\label{sys:KPP_truncated_domain}
    \end{equation}

    Furthermore, it satisfies
    \begin{equation}
	\underline{\veu}\vee\vez\leq\veu\leq\overline{\veu}\quad\text{in }\Sigma_a.
	\label{eq:profile_trapping_truncated_domain}
    \end{equation}
\end{prop}

\begin{proof}
    The existence of $a^\star$, and the fact that any
    value $a\geq a^\star$ actually satisfies the same condition \eqref{eq:values_of_a_large_enough}, 
    can be reformulated as the boundedness from below of the set
    \begin{equation*}
	\left\{ z\in\R\ |\ \exists(t,y)\in[0,T]\times[0,L_y]\quad \exists i\in[N]\quad \underline{u}_i(t,y,z)<0 \right\}.
    \end{equation*}
    This bound is a direct consequence of the definition of $\underline{\veu}$.

    The core of the proof will use the following two differential inequalities, valid for any 
    H\"{o}lder-continuous function $\ver\in[\vez,\overline{\veu}]$:
    \begin{equation}
	\left( \cbR+\diag\left( \veB'\ver \right) \right)\left( \underline{\veu}\vee\vez \right)\leq\vez\quad\text{in }\Sigma_a,
	\label{eq:generalized_subsolution_system}
    \end{equation}
    \begin{equation}
	\left( \cbR+\diag\left( \veB'\ver \right) \right)\overline{\veu}\geq\vez\quad\text{in }\Sigma_a.
	\label{eq:generalized_supersolution_system}
    \end{equation}
    These are consequences of obvious monotonicity properties of $\vev\mapsto\veB'\vev$ implied by \ref{ass:KPP} and of 
    Corollary \ref{cor:generalized_subsolution_system_truncated_domain} and 
    Lemma \ref{lem:supersolution_in_truncated_domain}.

    To begin the proof, we fix temporarily a H\"{o}lder-continuous function $\ver\in[\vez,\overline{\veu}]$ 
    and we study the well-posedness of:
    \begin{equation}
	\begin{dcases}
	    (\cbR + \diag(\veB'\ver))\veu = \vez & \text{in }\Sigma_a, \\
	    \vep(t,y,\pm a)=(\underline{\veu}\vee\vez)(t,y,\pm a) & \text{for all }(t,y)\in[0,T]\times[0,L_y], \\
	    (t,y)\mapsto[z\mapsto\vep(t,y,z)] & \text{is periodic}.
	\end{dcases}
	\label{sys:linear_system_truncated_domain_heterogeneous_boundary_conditions}
    \end{equation}

    Since it is more convenient to manipulate homogeneous Dirichlet boundary conditions, let us transform the system
    into an equivalent one.
    By definition of $\overline{\veu}$ and $\underline{\veu}$, there exists 
    $f\in\caC^\infty\left([-a,a],[0,+\infty)\right)$ such that $f(-a)=1$, $f(a)=0$ and,
    for all $(t,y,z)\in\overline{\Sigma_a}$, 
    \begin{equation*}
	(\underline{\veu}\vee\vez)(t,y,z)\leq f(z)\underline{\veu}(t,y,-a)\leq \frac12\left((\underline{\veu}\vee\vez)(t,y,z)+\overline{\veu}(t,y,z)\right).
    \end{equation*}
    By setting $\vef=f\underline{\veu}(-a)$ and changing $\veu$ into $\widetilde{\veu}=\veu-\vef$, we find that $\veu$ solves 
    \eqref{sys:linear_system_truncated_domain_heterogeneous_boundary_conditions} if and only if $\widetilde{\veu}$ solves:
    \begin{equation}
	\begin{dcases}
	    (\cbR + \diag(\veB'\ver))\widetilde{\veu} = - (\cbR+\diag(\veB'\ver))\vef & \text{in }\Sigma_a, \\
	    \widetilde{\veu}(t,y,\pm a)=\vez & \text{for all }(t,y)\in[0,T]\times[0,L_y], \\
	    (t,y)\mapsto[z\mapsto\widetilde{\veu}(t,y,z)] & \text{is periodic}.
	\end{dcases}
	\label{sys:linear_system_truncated_domain_homogeneous_boundary_conditions}
    \end{equation}

    In order to verify that this new system is well-posed, we have to verify the invertibility of the associated linear operator.
    Since the principal eigenvalue of a monotone linear operator is the leftmost point of its spectrum in $\mathbb{C}$, it is 
    sufficient to prove that this principal eigenvalue is positive. Denote $\lambda_{1,\upp-\upDir}(\cbR+\diag(\veB'\ver))$ 
    the principal eigenvalue of $\cbR+\diag(\veB'\ver)$ with time periodicity and periodic--Dirichlet boundary conditions in space.     
    Assume by contradiction 
    \begin{equation*}
	\lambda_{1,\upp-\upDir}(\cbR+\diag(\veB'\ver))\leq 0.
    \end{equation*}
    On one hand, let $\vect{\varphi}$ be its associated periodic--Dirichlet principal eigenfunction. It satisfies
    \begin{equation*}
	(\cbR+\diag(\veB'\ver))\vect{\varphi}=\lambda_{1,\upp-\upDir}(\cbR+\diag(\veB'\ver))\vect{\varphi}\leq\vez\quad\text{in }\Sigma_a.
    \end{equation*}
    On the other hand, the function $\overline{\veu}\gg\vez$ satisfies \eqref{eq:generalized_supersolution_system}.
    Therefore, for a sufficiently small $\kappa>0$, the following inequalities hold true in $\Sigma_a$:
    \begin{equation*}
	\vez\leq\kappa\vect{\varphi}\leq\overline{\veu}\quad\text{and}\quad
	(\cbR+\diag(\veB'\ver))(\overline{\veu}-\kappa\vect{\varphi})\geq\vez.
    \end{equation*}
    For the supremum $\kappa^\star$ of such $\kappa$, there exists $(t^\star,y^\star,z^\star)\in\overline{\Sigma_a}$
    and $i^\star\in[N]$ such that $\overline{u}_{i^\star}(t^\star,y^\star,z^\star)-\kappa^\star\varphi_{i^\star}(t^\star,y^\star,z^\star)=0$.
    In view of the boundary conditions, it occurs in fact at $(t^\star,y^\star,z^\star)\in\Sigma_a$.
    Then the strong comparison principle for linear cooperative operators and the continuity show that 
    $\kappa^\star\vect{\varphi}=\overline{\veu}$ in $\overline{\Sigma_a}$, a contradiction with the boundary conditions. 
    Hence $\lambda_{1,\upp-\upDir}(\cbR+\diag(\veB'\ver))>0$. 

    By \eqref{eq:generalized_subsolution_system}, the function $\underline{\veu}\vee\vez-\vef\leq\vez$ 
    is a sub-solution of \eqref{sys:linear_system_truncated_domain_homogeneous_boundary_conditions}. 
    By \eqref{eq:generalized_supersolution_system}, the function $\overline{\veu}-\vef\geq\vez$ is a 
    super-solution of \eqref{sys:linear_system_truncated_domain_homogeneous_boundary_conditions}. 
    By monotonicity of the operator, the unique classical solution $\widetilde{\veu}$ of 
    \eqref{sys:linear_system_truncated_domain_homogeneous_boundary_conditions} satisfies
    $\underline{\veu}\vee\vez-\vef\leq\widetilde{\veu}\leq\overline{\veu}-\vef$ in $\Sigma_a$. 

    Therefore, the equivalent problem \eqref{sys:linear_system_truncated_domain_heterogeneous_boundary_conditions}
    is well-posed, and its solution, denoted $\veu_{\ver}$, satisfies
    $\underline{\veu}\vee\vez\leq\veu_{\ver}\leq\overline{\veu}$ in $\Sigma_a$.

    Subsequently, the mapping $\vect{\Phi}:\ver\mapsto\veu_{\ver}$ defined on the following convex set:
    \begin{equation*}
	\mathcal{F} = \left\{ \ver\in\caC^{\delta/2,\delta}_{\upp}\left(\R\times\R^{n-1},\caC^{\delta}\left( [-a,a],\R^N \right)\right)\ |\ \underline{\veu}\vee\vez\leq\ver\leq\overline{\veu} \right\}
    \end{equation*}
    is well-defined and satisfies $\vect{\Phi}(\mathcal{F})\subset\mathcal{F}$.
    Furthermore, by classical parabolic estimates and compact embeddings of H\"{o}lder spaces, $\vect{\Phi}$ is a compact mapping.
    Therefore, by virtue of the Schauder fixed point theorem, there exists a solution $\veu$ of 
    \eqref{sys:KPP_truncated_domain}, that satisfies \eqref{eq:profile_trapping_truncated_domain}.
\end{proof}

\subsection{System in the non-truncated cylinder}
In this subsection, we pass to the limit $a\to+\infty$ and obtain an entire solution of $\cbR\veu=-\veB'\veu\circ\veu$.
Recall that $\Sigma_a=(0,T)\times(0,L_y)\times(-a,a)$.

To do so, we first need a pointwise bound uniform with respect to $a\geq a^\star$, where $a^\star$ is defined
in Proposition \ref{prop:existence_profile_in_truncated_domain}.
To emphasize its dependency on $a$, we denote $\veu_a$ the solution of 
\eqref{sys:KPP_truncated_domain} constructed in the previous subsection.

\begin{lem}\label{lem:uniform_pointwise_bound}
    There exists $K>0$, dependent only on bounds on $\veL$ and $\veB$, such that, up to increasing $a^\star$ in a 
    way that only depends on $K$ and $\muwed$,
    \begin{equation}
	0\leq\veu_a\leq K\veo\quad\text{in }\overline{\Sigma_a}\quad\text{for all }a\geq a^\star.
	\label{eq:uniform_pointwise_bound}
    \end{equation}
\end{lem}

\begin{proof}
    The proof relies upon the exact same idea as the one in \cite[Proposition 2.1]{Girardin_2023} and is 
    actually an \textit{a priori} estimate on solutions of \eqref{sys:KPP_truncated_domain}.
    For the sake of completeness and to show precisely how boundary conditions are accounted for,
    we recall the proof.

    By assumptions \ref{ass:smooth_periodic} and \ref{ass:KPP}, there exist constants
    $r,K>0$ such that, for any $\veu\geq\vez$,
    \begin{equation*}
	\veL'\veu-\veB'\veu\circ\veu\leq r\left( \veo^\upT\veu \right)\left( K\veo-\veu \right).
    \end{equation*}
    By definition of $\underline{\veu}$, it satisfies $\underline{\veu}(t,y,z)\leq\upe^{\muwed z}\veo$
    globally in $\R\times\R^{n-1}\times\R$. Therefore, up to increasing $a^\star$ in an obvious way, for any $a\geq a^\star$, 
    $\underline{\veu}(t,y,-a)\leq K\veo$ in $[0,T]\times[0,L_y]$.

    To ease the reading we denote, for each $i\in[N]$,
    \begin{equation*}
	\widetilde{\caP}_i = (\cbR+\veL')_i = \partial_t-\nabla\cdot(A_i'\nabla)+q_i'\cdot\nabla
    \end{equation*}

    Solutions $\veu$ of \eqref{sys:KPP_truncated_domain} satisfy, for each $i\in[N]$,
    \begin{equation*}
	\widetilde{\caP}_i u_i \leq r(K-u_i)\sum_{j=1}^N u_j.
    \end{equation*}
    Whenever $u_i\geq K$,
    \begin{equation*}
	\widetilde{\caP}_i u_i \leq ru_i(K-u_i).
    \end{equation*}
    In particular, $v_i=\max(u_i,K)$ is a sub-solution of the periodic--Dirichlet semilinear scalar problem:
    \begin{equation}
	\begin{dcases}
	    \widetilde{\caP}_i u = ru(K-u) & \text{in }\Sigma_a, \\
	    u(t,y,\pm a) = K & \text{for any }(t,y)\in[0,T]\times[0,L_y], \\
	    (t,y)\mapsto[z\mapsto u(t,y,z)] & \text{is periodic}.
	\end{dcases}
	\label{eq:per_Dir_scalar_problem_uniform_pointwise_bound}
    \end{equation}

    Let $\widetilde{K}=\max_{\overline{\Sigma_a}}v_i$ and assume by contradiction $\widetilde{K}>K$. 
    The constant $\widetilde{K}$ is a super-solution of 
    \eqref{eq:per_Dir_scalar_problem_uniform_pointwise_bound} and, by definition, it satisfies $\widetilde{K}\geq v_i$.
    Also by definition, it satisfies $\widetilde{K}=v_i(t^\star,y^\star,z^\star)$ for certain 
    $(t^\star,y^\star,z^\star)\in\overline{\Sigma_a}$. In view of the boundary conditions, $z^\star\in(-a,a)$.
    By virtue of the scalar comparison principle \cite{Protter_Weinberger}, 
    $v_i=\widetilde{K}$ globally in $\overline{\Sigma_a}$. But this contradicts the boundary conditions
    satisfied by $v_i$. Hence $\widetilde{K}=K$.
    This ends the proof.
\end{proof}

Now, for all $a\geq a^\star$, we extend $\veu_a$ in $\R\times\R^{n-1}\times\R$ by setting $\veu_a=\vez$
outside $\R\times\R^{n-1}\times[-a,a]$. This extension is continuous and periodic with respect to $(t,y)$. It also satisfies
the estimate \eqref{eq:uniform_pointwise_bound}.

Thanks to this estimate, classical parabolic estimates \cite{Lieberman_2005} and a diagonal extraction process, 
we can extract a sequence $\left(a_k,\veu_{a_k}\right)_{k\in\N}$ such that $a_k\to+\infty$ and
$\veu_{a_k}$ converges in $\caC^{1,2}_{\upp}(\R\times\R^{n-1},\caC^2_{\upl}(\R,\R^N))$ to
some limit $\veu'$ satisfying $\underline{\veu}\vee\vez\leq\veu'\leq\overline{\veu}\wedge(K\veo)$
globally, where $\overline{\veu}\wedge(K\veo)$ denotes the component-by-component minimum of 
$\overline{\veu}$ and $K\veo$. For future reference, this is stated in the following corollary.

\begin{cor}
    There exists a solution $\veu'\in\caC^{1,2}_{\upp}(\R\times\R^{n-1},\caC^2(\R,\R^N))$ of 
    \begin{equation}
	\cbR\veu' = -(\veB'\veu')\circ\veu' \quad\text{in }\R\times\R^{n-1}\times\R.
	\label{sys:profile_system_full_domain}
    \end{equation}
    
    Furthermore, it satisfies
    \begin{equation}
	\underline{\veu}\vee\vez\leq\veu'\leq\overline{\veu}\wedge(K\veo)\quad\text{in }\R\times\R^{n-1}\times\R.
	\label{eq:profile_trapping_full_domain}
    \end{equation}
\end{cor}

\subsection{Upstream positivity}
In this subsection, we prove the asymptotic uniform positivity of the profile, \revision{the limit} 
\eqref{eq:definition_GTW_limit_positive_infinity} \revision{in Definition \ref{defi:GTW}}. 
\revision{In view of Proposition \ref{prop:transformed_WPPTW_problem},}
this will end the proof of Theorem \ref{thm:supercritical_existence}.

\begin{prop}
    Necessarily,
    \begin{equation*}
	\liminf_{z\to+\infty}\min_{(t,y)\in[0,T]\times[0,L_y]}\min_{i\in[N]}u_i'(t,y,z)>0.
    \end{equation*}
\end{prop}

\begin{proof}
    Assume by contradiction
    \begin{equation}\label{eq:upstream_contradiction}
	\liminf_{z\to+\infty}\min_{(t,y)\in[0,T]\times[0,L_y]}\min_{i\in[N]}u_i'(t,y,z)=0.
    \end{equation}

    Let $(i_k,t_k,y_k)_{k\in\N}\in([N]\times[0,T]\times[0,L_y])^\N$ and $(z_k)_{k\in\N}$ such that $z_k\to+\infty$
    and $u_{i_k}'(t_k,y_k,z_k)\to 0$ as $k\to+\infty$. Up to extraction, 
    $(i_k,t_k,y_k)\to(i_\infty,t_\infty,y_\infty)\in[N]\times[0,T]\times[0,L_y]$.
    Moreover, let $Z_k\in[0,L_z]$ such that, for each $k\in\N$, $\zeta_k=z_k-Z_k\in L_z\mathbb{Z}$.
    Up to another extraction, $(\zeta_k)_{k\in\N}$ is increasing and $Z_k\to Z_\infty\in[0,L_z]$. It follows by continuity that 
    \begin{equation*}
	u_{i_\infty}'(t_\infty,y_\infty,\zeta_k+Z_\infty)\to 0\quad\text{as }k\to+\infty.
    \end{equation*}
    Applying the Harnack inequality of Proposition \ref{prop:harnack_inequality} to 
    $(t,y,z)\mapsto\veu'(t+t_\infty,y+y_\infty,z+\zeta_k+Z_\infty)$
    and to the translated operator, denoted for simplicity $\cbR(t+t_\infty,y+y_\infty,z+Z_\infty)+\diag(\veB'\veu')(t+t_\infty,y+y_\infty,z+Z_\infty)$,
    we easily deduce that
    \begin{equation*}
	\max_{i\in[N]}u_i'(t,y,z+\zeta_k)\to 0
    \end{equation*}
    uniformly with respect to $(t,y,z)\in[0,T]\times[0,L_y]\times[0,L_z]$. 

    In general this shows that limit points as $z\to+\infty$ are either $\vez$ or in $(\vez,K\veo)$.
    By \eqref{eq:upstream_contradiction}, here it is $\vez$. Thanks to this convergence, 
    we are now going to replace the semilinear system by a linearized one.

    For any $i\in[N]$ and any $k\in\N$, the function 
    \begin{equation*}
	\vev_k:(t,y,z)\mapsto\frac{1}{u_i'(t_\infty,y_\infty,\zeta_k+Z_\infty)}\veu'(t+t_\infty,y+y_\infty,z+\zeta_k+Z_\infty)
    \end{equation*}
    satisfies, for all $(t,y,z)\in\R\times\R^{n-1}\times\R$ and $k\in\N$,
    \begin{equation*}
	\cbR(t+t_\infty,y+y_\infty,z+z_\infty)\vev_k(t,y,z) = -(\veB'\veu')(t+t_\infty,y+y_\infty,z+\zeta_k+Z_\infty)\circ\vev_k(t,y,z),
    \end{equation*}
    or else
    \begin{equation*}
	\cbR(t,y,z)\vev_k(t-t_\infty,y-y_\infty,z-Z_\infty) = -(\veB'\veu')(t,y,z)\circ\vev_k(t-t_\infty,y-y_\infty,z-Z_\infty),
    \end{equation*}
    as well as $v_{k,i}(0,0,0)=1$.
    By classical parabolic estimates \cite{Lieberman_2005} and a diagonal extraction process, the sequence 
    $(\vev_k(\cdot-t_\infty,\cdot-y_\infty,z-Z_\infty))_{k\in\N}$
    converges, up to extraction, locally uniformly, to a nonnegative nonzero solution 
    $\vev\in\caC^{1,2}_{\upp}(\R\times\R^{n-1},\caC^2(\R,\R^N))$ of
    \begin{equation*}
	\cbR\vev=\vez.
    \end{equation*}

    In other words, $\vev$ is a generalized principal eigenfunction for the generalized principal eigenvalue $0$
    of the operator $\cbR$ \cite{Girardin_Mazari_2022}. 
    By the Harnack inequality of Proposition \ref{prop:harnack_inequality}, it is positive.

    Now we are going to adapt the proof of \cite[Proposition 3.9]{Girardin_Mazari_2022} to prove an exponential estimate, that will lead to a contradiction.

    Define the translation $\tau:z\in\R\mapsto z+L_z$ and denote $\vev^\tau:(t,y,z)\mapsto\vev(t,y,\tau(z))$ and 
    $\vew=\left(v^\tau_i/v_i\right)_{i\in[N]}$. By virtue of the Harnack inequality of Proposition \ref{prop:harnack_inequality} and
    periodicity of the coefficients of $\cbR$, $\vew$ is globally bounded.
    Let 
    \[
	\overline{\mu}=\frac{1}{L_z}\ln\left(\max_{i\in[N]}\sup_{(t,y,z)\in\R\times\R^{n-1}\times\R}w_i(t,y,z)\right).
    \]
    Recalling that $\vev$ and consequently $\vew$ are periodic with respect to all variables except maybe $z$,
    there exists $\overline{i}\in[N]$ and a new sequence still denoted 
    $\left(t_k,y_k,z_k\right)_{k\in\N}\in([0,T]\times[0,L_y]\times\R)^\N$ such that
    $w_{\overline{i}}(t_k,y_k,z_k)\to\upe^{\overline{\mu} L_z}$ as $k\to+\infty$. 
    Moreover, there exists another sequence still denoted $\left(Z_k\right)_{k\in\N}\in[0,L_z]^{\N}$
    such that, for all $k\in\N$, $z_k-Z_k\in L_z\mathbb{Z}$. Up to extraction,
    we assume that $(t_k,y_k,Z_k)\to (t_\infty,y_\infty,Z_\infty)\in[0,T]\times[0,L_y]\times[0,L_z]$.
    
    Now, define, for all $k\in\N$,
    \[
        \hat{\vev}_k:(t,y,z)\mapsto\frac{1}{v_{\overline{i}}(t_k,y_k,z_k)}\vev(t+t_k,y+y_k,z+z_k),
    \]
    \[
        \hat{\vev}^\tau_k:(t,y,z)\mapsto\hat{\vev}_k(t,y,\tau(z)).
    \]
    
    Once more by virtue of the Harnack inequality and
    the periodicity of the coefficients of $\cbR$, $\left(\hat{\vev}_k\right)_{k\in\N}$ is globally bounded.
    By periodicity of the coefficients of $\cbR$, it satisfies:
    \[
        \cbR(t+t_k,y+y_k,z+Z_k)\hat{\vev}_k(t,y,z)=\vez\quad\text{for all }(t,y,z)\in\R\times\R^{n-1}\times\R,\ k\in\N.
    \]
    Therefore, by classical regularity estimates \cite{Lieberman_2005}, $\left(\hat{\vev}_k\right)_{k\in\N}$
    converges up to a diagonal extraction to $\hat{\vev}_\infty\in\caC^{1,2}_{\upp}(\R\times\R^{n-1},\caC^2(\R,\R^N))$
    which satisfies:
    \[
        \cbR(t+t_\infty,y+y_\infty,z+Z_\infty)\hat{\vev}_\infty(t,y,z)=\vez\quad\text{for all }(t,y,z)\in\R\times\R^{n-1}\times\R,
    \]
    or else:
    \[
        \cbR(t,y,z)\hat{\vev}_\infty(t-t_\infty,y-y_\infty,z-Z_\infty)=\vez\quad\text{for all }(t,y,z)\in\R\times\R^{n-1}\times\R.
    \]
    Moreover, $\hat{v}_{\overline{i},\infty}(0,0,0)=1$, whence $\hat{\vev}_\infty$ is nonzero.
    By the strong maximum principle, it is in fact positive.
    
    We can now define $\vef_\infty=\upe^{\overline{\mu} L_z}\hat{\vev}_\infty -\hat{\vev}_\infty^\tau$.
    We deduce by definition of $\overline{\mu}$ that $\vef_\infty\geq\vez$ with $f_{\overline{i},\infty}(0,0,0)=0$. 
    Moreover, $\vef_\infty$ satisfies the same equation as $\hat{\vev}_\infty$. Therefore,
    by virtue of the strong maximum principle, $\vef_\infty$ is the zero function.
    This exactly means that $\upe^{\overline{\mu} L_z}\hat{\vev}_\infty=\hat{\vev}_\infty^\tau$.
    
    It is now clear that $\vev^1:(t,y,z)\mapsto\hat{\vev}_\infty(t-t_\infty,y-y_\infty,z-Z_\infty)$ is positive, 
    periodic with respect to $t$ and the $n-1$ first spatial variables, is a solution of $\cbR\vev^1=\vez$, and that the function
    $(t,y,z)\mapsto\upe^{-\overline{\mu} z}\vev^1(t,y,z)$ is $L_z$-periodic with respect to $z$. 

    This implies that $0=\lambda_{1,\upp}(\cbR_{\overline{\mu}})$. Applying \eqref{eq:principal_eigenfunction_new_coordinates},
    it follows that $0=\lambda_{1,\overline{\mu} e}+c\overline{\mu}$. 
    The function $\nu\in\R\mapsto\lambda_{1,\nu e}+c\nu$ is strictly concave and coercive (\textit{cf.} \cite[Lemma 3.5]{Girardin_2023}),
    and thus $\muwed$ and $\muvee$ given by Lemma \ref{lem:minimal_c} are its only two zeros.
    Hence $\overline{\mu}\in\{\muwed,\muvee\}$.

    Repeating the same process with $\overline{\mu}$ replaced by
    \begin{equation*}
	\underline{\mu}=\frac{1}{L_z}\ln\left(\min_{i\in[N]}\inf_{(t,y,z)\in\R\times\R^{n-1}\times\R}w_i(t,y,z)\right),
    \end{equation*}
    we end up similarly with $\underline{\mu}\in\{\muwed,\muvee\}$.

    Hence
    \begin{equation*}
	\upe^{\muwed L_z}\veo\leq\vew\leq\upe^{\muvee L_z}\veo\quad\text{in }\R\times\R^{n-1}\times\R.
    \end{equation*}

    By definition of $\vew$ and $\vev^\tau$, for all $(t,y,z)\in\R\times\R^{n-1}\times\R$,
    \begin{equation*}
	\upe^{\muwed L_z}\vev(t,y,z)\leq\vev(t,y,z+L_z)\leq\upe^{\muvee L_z}\vev(t,y,z).
    \end{equation*}

    By definition of $\vev$ and by its positivity and global boundedness, there exists $k_0\in\N$ such that,
    for all $(t,y,z)\in[0,T]\times[0,L_y]\times[0,L_z]$ and all $k\in\N$,
    \begin{equation*}
	\upe^{\frac{\muwed}{2} L_z}\veu'(t,y,z+\zeta_{k+k_0})\leq\veu'(t,y,z+\zeta_{k+k_0}+L_z).
    \end{equation*}
    Recall that, by construction, for all $k\in\N$, $\zeta_k\in L_z\Z$ and $\zeta_k<\zeta_{k+1}$.
    By iteration,
    for all $(t,y,z)\in[0,T]\times[0,L_y]\times[0,L_z]$ and all $k\in\N$,
    \begin{equation*}
	\upe^{(\zeta_{k+k_0+1}-\zeta_{k+k_0})\frac{\muwed}{2} L_z}\veu'(t,y,z+\zeta_{k+k_0})\leq\veu'(t,y,z+\zeta_{k+k_0+1}).
    \end{equation*}
    Subsequently, after another iteration, for all $(t,y,z)\in[0,T]\times[0,L_y]\times[0,L_z]$ and all $k\in\N$,
    \begin{equation}\label{eq:upstream_contradiction_exponential_estimate}
	\upe^{(\zeta_{k+k_0}-\zeta_{k_0})\frac{\muwed}{2} L_z}\veu'(t,y,z+\zeta_{k_0})\leq\veu'(t,y,z+\zeta_{k+k_0}).
    \end{equation}

    However the exponential estimate \eqref{eq:upstream_contradiction_exponential_estimate}, with $\muwed>0$ (\textit{cf.} Lemma \ref{lem:minimal_c}), obviously contradicts the convergence $\veu'(t,y,z+\zeta_k)\to\vez$ as $k\to+\infty$. Thus \eqref{eq:upstream_contradiction} is contradicted and, consequently, the limit \eqref{eq:definition_GTW_limit_positive_infinity} is established.
\end{proof}

\section{Proof of Theorem \ref{thm:critical_existence}}

In this section, we assume again $\lambda_{1,\upp}<0$, we fix $e\in\Q^n\cap\Sn$, we assume $c_e^\star\in\Q$, and
we prove the existence of a pulsating traveling wave. Again, we omit the subscripts $e$ and $c^\star_e$ in all forthcoming notations
and we use the notations $T$ and $L$ for the time and space periods of the coefficients of the rotated operator $\cbR$.

The proof is quite similar to that of Theorem \ref{thm:supercritical_existence}. In fact, except minor obvious adaptations,
the only substantial difference is in the construction
of $\overline{\veu}$ and $\underline{\veu}$ before proving Proposition \ref{prop:existence_profile_in_truncated_domain}.
For the sake of brevity, below, we only detail these two new constructions in the truncated cylinder
$\Sigma_a = (0,T)\times(0,L_y)\times(-a,a)$. 

The constructions will use repeatedly the mapping
\begin{equation*}
    \vect{\Theta}:\mu\in(0,+\infty)\mapsto\left[(t,y,z)\mapsto\upe^{\mu z}\veu_{\mu}'(t,y,z)\right]. 
\end{equation*}
By the Krein--Rutman theorem,
the principal eigenvalue $\lambda_{1,\mu}=\lambda_{1,\upp}(\cbQ_{\mu})$ is always simple, and by the Kato--Rellich theorem, it depends smoothly
on $\mu$. It follows that $\vect{\Theta}$ is of class $\caC^1$. Below, we denote its derivative
$\dot{\vect{\Theta}}$.

We also need two new notations to circumvent possible oscillations of the functions at hand.
If $\vev$ is a function of the variable $z\in\R$ of class $\caC^1$ such that, for each $i\in[N]$,
there exists $z_i\in\R$ such that $v_i(z)>0$ in $(-\infty,z_i)$, $v_i(z_i)=0$ and $v_i'(z_i)<0$, then $\vev\vee_{-\infty}\vez$
denotes the continuous, piecewise $\caC^1$, function whose $i$-th component equals $v_i$ in $(-\infty,z_i)$ and $0$ in $[z_i,+\infty)$.
Consistently, $\vev\wedge_{-\infty}\veo$ denotes the function $\veo - (\veo-\vev)\vee_{-\infty}\vez$.

\begin{lem}
    Let $M_1>0$, $M_2>0$ and
    \begin{equation}
	\overline{\veu}=M_2\left(\left(-M_1\dot{\vect{\Theta}}(\mu^\star)\right)\wedge_{-\infty}\veo\right).
	\label{defi:critical_supersolution}
    \end{equation}

    If $M_1$ is large enough, then $\overline{\veu}\geq\vez$ in $\R\times\R^{n-1}\times\R$.

    Furthermore, if $M_2$ is large enough, then $\overline{\veu}$ satisfies:
    \begin{equation}
	\cbR\overline{\veu}\geq -\diag(b'_{i,i})\overline{\veu}^{\circ 2}\quad\text{in }\Sigma_a
	\label{eq:critical_supersolution_in_truncated_domain}
    \end{equation}
    where partial derivatives are understood in distributional sense where necessary.
\end{lem}

\begin{proof}
    On one hand, differentiating the identity \eqref{eq:principal_eigenfunction_new_coordinates} with respect to $\mu$, we find directly
    \begin{equation}\label{eq:differentiation_eigen_equality_wrt_mu}
	\cbR\dot{\vect{\Theta}}(\mu) = \left( \frac{\partial\lambda_{1,\mu}}{\partial\mu} +c^\star \right)\vect{\Theta}(\mu) + \left( \lambda_{1,\mu}+c^\star\mu \right)\dot{\vect{\Theta}}(\mu).
    \end{equation}

    Note that the derivative of $\mu\in(0,+\infty)\mapsto\frac{-\lambda_{1,\mu}}{\mu}$ is 
    $\mu\in(0,+\infty)\mapsto\frac{-1}{\mu}\frac{\partial\lambda_{1,\mu}}{\partial\mu}+\frac{\lambda_{1,\mu}}{\mu^2}$.
    In particular, evaluating the derivative at $\mu=\mu^\star>0$ where the global strict minimum 
    $\frac{\lambda_{1,\mu^\star}}{\mu^\star}=c^\star$ is achieved, we deduce:
    \begin{equation}
        c^\star = -\left(\frac{\partial\lambda_{1,\mu}}{\partial\mu}\right)_{|\mu=\mu^\star}.
        \label{eq:zero_derivative_at_cstar}
    \end{equation}

    Therefore, evaluating \eqref{eq:differentiation_eigen_equality_wrt_mu} at $\mu^\star$ and using the linearity,
    \begin{equation*}
	\cbR(-M_1\dot{\vect{\Theta}}(\mu^\star)) = \vez \geq -\diag(b'_{i,i})\left( -M_1\dot{\vect{\Theta}}(\mu^\star) \right)^{\circ 2}.
    \end{equation*}

    On the other hand, if $M_2$ is large enough, then $-\veL'\veo + M_2\diag(b'_{i,i})\veo\geq\vez$, whence
    $\cbR(M_2\veo)\geq -\diag(b'_{i,i})(M_2\veo)^{\circ 2}$.

    As a matter of fact, for any $(t,y,z)\in\R\times\R^{n-1}\times\R$
    \begin{equation}
	\dot{\vect{\Theta}}(\mu^\star)(t,y,z) = \upe^{\mu^\star z}\left(z\veu'_{\mu^\star}(t,y,z)+\left(\frac{\partial\veu'_{\mu}}{\partial\mu}\right)_{|\mu=\mu^\star}(t,y,z)\right),
	\label{eq:expansion_derivative_mu_supersolution}
    \end{equation}
    so that $-M_1\dot{\vect{\Theta}}(\mu^\star)$ is positive in a neighborhood of $z=-\infty$. Increasing appropriately $M_1$,
    it follows that $(-M_1\dot{\vect{\Theta}}(\mu^\star))\wedge_{-\infty}\veo$ is well-defined and positive in $\R\times\R^{n-1}\times\R$.

    Therefore $M_2((-M_1\dot{\vect{\Theta}}(\mu^\star))\wedge_{-\infty}\veo)$
    is also positive, and finally it satisfies indeed \eqref{eq:critical_supersolution_in_truncated_domain} in distributional sense.
\end{proof}

Below, the two constants $M_1$ and $M_2$ are fixed accordingly.

\begin{lem}
    Let $M_3>0$, $\gamma=\frac{\mu^\star}{2}$ and
    \begin{equation}
	\underline{\veu}= M_1 M_2\left(\left(-\dot{\vect{\Theta}}(\mu^\star)-M_3\vect{\Theta}(\mu^\star)+\vect{\Theta}(\mu^\star+\gamma)\right)\vee_{-\infty}\vez\right).
	\label{defi:critical_subsolution}
    \end{equation}

    If $M_3$ is large enough, then $\underline{\veu}$ satisfies:
    \begin{equation}\label{eq:ordered_critical_subsolution_supersolution_in_truncated_domain}
	\underline{\veu}\leq\overline{\veu}\quad\text{in }\R\times\R^{n-1}\times\R,
    \end{equation}
    \begin{equation}
	\cbR\underline{\veu}\leq -\underline{\veu}\circ(\veB'\overline{\veu})\quad\text{in }\Sigma_a
	\label{eq:critical_subsolution_in_truncated_domain}
    \end{equation}
    where partial derivatives are understood in distributional sense where necessary.
\end{lem}

\begin{proof}
    Evaluating at $z=0$ and $(t,y)\in\R\times\R^{n-1}$, 
    \begin{equation*}
	\underline{\veu}(t,y,0) = M_1 M_2\left( \left( -\left(\frac{\partial\veu'_{\mu}}{\partial\mu}\right)_{|\mu=\mu^\star}(t,y,0)-M_3\veu'_{\mu^\star}(t,y,0)+\veu'_{\mu^\star+\gamma}(t,y,0)\right)\vee_{-\infty}\vez \right),
    \end{equation*}
    whence there exists $M_3^0$ such that, provided $M_3\geq M_3^0$, $\underline{\veu}(t,y,0)=\vez$. Subsequently,
    by definition of $\vee_{-\infty}$, $M_3\geq M_3^0$ implies $\underline{\veu}=\vez$ in $\R\times\R^{n-1}\times[0,+\infty)$.

    Also, by positivity of $\gamma$ and up to increasing again $M_3$, 
    $-M_3\vect{\Theta}(\mu^\star)+\vect{\Theta}(\mu^\star+\gamma)\leq\vez$ in 
    $\R\times\R^{n-1}\times(-\infty,0]$. Hence $\underline{\veu}\leq -M_1 M_2 \dot{\vect{\Theta}}(\mu^\star)$ in 
    $\R\times\R^{n-1}\times(-\infty,0]$.

    Up to increasing once more $M_3$, $\underline{\veu}\leq M_2\veo$ in $\R\times\R^{n-1}\times\R$, whence
    \eqref{eq:ordered_critical_subsolution_supersolution_in_truncated_domain} is proved.

    It remains to verify \eqref{eq:critical_subsolution_in_truncated_domain} provided $M_3$ is sufficiently large.
    In fact, it is sufficient to verify that $M_3$ can be chosen so large that, in $\R\times\R^{n-1}\times(-\infty,0]$,
    \begin{equation}
	(\cbR+\diag(\veB'\overline{\veu}))\left( -\dot{\vect{\Theta}}(\mu^\star)-M_3\vect{\Theta}(\mu^\star)+\vect{\Theta}(\mu^\star+\gamma)\right) \leq \vez.
	\label{eq:critical_subsolution_left_halfline}
    \end{equation}

    By \eqref{eq:zero_derivative_at_cstar} and the Taylor theorem,
    \begin{equation}
	\lambda_{1,\mu^\star+\gamma}=-c^\star\mu^\star-c^\star\gamma+\int_{\mu^\star}^{\mu^\star+\gamma}(\mu^\star+\gamma-\mu)\frac{\partial^2\lambda_{1,\mu}}{\partial\mu^2}\upd\mu.
        \label{eq:second_order_Taylor_expansion_lambda1mu}
    \end{equation}
    By virtue of the concavity of $\mu\mapsto\lambda_{1,\mu}$ (\textit{cf.} \cite[Theorem 1.3]{Girardin_Mazari_2022}), the integral remainder is nonpositive. 
    The concavity being in fact strict (\textit{cf.} \cite[Theorem 1.3]{Girardin_Mazari_2022}), it is zero if and only if $\gamma=0$. Since $\gamma=\mu^\star/2$, this is not the case.
    In other words, there exists a strictly concave function $g:\R\to\R$ such 
    that $g(0)=g'(0)=0$, $g(\gamma)<0$ and
    $\lambda_{1,\mu^\star+\gamma}\leq -c^\star\mu^\star-c^\star\gamma+g(\gamma)$.
    
    Hence we deduce from \eqref{eq:principal_eigenfunction_new_coordinates} that
    \begin{equation}
	\cbR\vect{\Theta}(\mu^\star+\gamma) \leq g(\gamma)\vect{\Theta}(\mu^\star+\gamma).
        \label{eq:Taylor_expansion_subsolution}
    \end{equation}
    This inequality does not depend on $M_1$, $M_2$ or $M_3$.

    By linearity, $\cbR\dot{\vect{\Theta}}(\mu^\star)=\vez$, $\cbR\vect{\Theta}(\mu^\star)=\vez$ and \eqref{eq:Taylor_expansion_subsolution},
    \eqref{eq:critical_subsolution_left_halfline} reduces to:
    \begin{equation}
	g(\gamma)\vect{\Theta}(\mu^\star+\gamma)+\left(\veB'\overline{\veu}\right)\circ\left( -\dot{\vect{\Theta}}(\mu^\star)-M_3\vect{\Theta}(\mu^\star)+\vect{\Theta}(\mu^\star+\gamma)\right) \leq \vez.
	\label{eq:reduced_critical_subsolution_left_halfline}
    \end{equation}

    On one hand, recall the definition of $\overline{\veu}$, \eqref{defi:critical_supersolution}, as well as the expansion
    \eqref{eq:expansion_derivative_mu_supersolution}. Considering the asymptotics as $z\to-\infty$, we find
    that, omitting the arguments in $(t,y)$,
    \begin{equation*}
	\left(\veB'\overline{\veu}\right)\circ\left( -\dot{\vect{\Theta}}(\mu^\star)-M_3\vect{\Theta}(\mu^\star)+\vect{\Theta}(\mu^\star+\gamma)\right)(z)
	\underset{z\to -\infty}{\sim} z^2\upe^{2\mu^\star z}\veB'\veu'_{\mu^\star}(z)\circ\veu'_{\mu^\star}(z).
    \end{equation*}
    Therefore, as $z\to-\infty$, this term is negligible in \eqref{eq:reduced_critical_subsolution_left_halfline} provided 
    $\mu^\star+\gamma<2\mu^\star$. Recalling again $\gamma=\mu^\star/2$, this is true indeed.

    On the other hand, away from $z=-\infty$, it suffices to increase again $M_3$
    to ensure the negativity of the left-hand side in \eqref{eq:reduced_critical_subsolution_left_halfline}.

    This ends the proof.
\end{proof}

\begin{rem}
    Contrarily to the construction in the super-critical case $c>c^\star$, here in the critical case the super-solution accounts
    for a certain amount of nonlinearity (the quadratic term in \eqref{eq:critical_supersolution_in_truncated_domain}). 
    This is required in order to circumvent the sign changes of $-\dot{\vect{\Theta}}(\mu^\star)$.
    Yet this nonlinearity has no coupling effect on the components of $\veu$, and therefore it preserves the maximum
    principle for cooperative systems. Thanks to this key fact, the rest of the proof is not altered significantly. 
    In the proof of Proposition \ref{prop:existence_profile_in_truncated_domain}, it suffices
    to replace the cooperative linear operator $\cbR+\diag(\veB'\ver)$ by the cooperative semilinear operator
    $\veu\mapsto\cbR\veu+\diag\left(b'_{i,i}\right)\veu^{\circ 2} +\diag\left( \left( \veB'-\diag\left(b'_{i,i}\right) \right)\ver \right)\veu$.
\end{rem}

\revision{\section{Proof of Corollary \ref{cor:space_homogeneous}}

\begin{proof}
Back to the beginning of the proof of Theorem \ref{thm:supercritical_existence}, we replace the    
first two paragraphs preceding Section \ref{sec:change_of_variables_and_coefficients} as follows.

\begin{quote}
In this section, we assume $\lambda_{1,\upp}<0$ and we prove the existence of pulsating traveling 
waves for all $e\in\Sn$ and $c\in(c^\star_e,+\infty)$.

The wave direction $e$ and the wave speed $c$ are fixed once and for all. 
In this whole section, time periodicity implicitly refers to $1$-periodicity and space periodicity
implicitly refers to space homogeneity.
\end{quote}

Then we replace the beginning of Section \ref{sec:change_of_variables_and_coefficients} up to Proposition
\ref{prop:transformed_WPPTW_problem} as follows.

\begin{quote}
Let $P\in\R^{n\times n}$ be the orthogonal matrix associated with the change of variables replacing $x$ (the coordinates
in the canonical basis of $\R^n$) into $x'$ (the coordinates in any new orthonormal basis $(e_\alpha)_{\alpha\in[n]}$ with
$e_n=e$), \textit{i.e.}
\begin{equation}
    \begin{pmatrix}x_1 \\ \vdots \\ x_{n-1} \\ x_n\end{pmatrix}
    = P\begin{pmatrix}x'_1 \\ \vdots \\ x'_{n-1} \\ x'_n\end{pmatrix}.
\end{equation}
Remarking that $\nabla_{x'} = P^{\upT}\nabla_x = P^{-1}\nabla_x$, we set
\begin{equation*}
    A_i':(t)\mapsto P^{\upT}A_i(t)P,
\end{equation*}
\begin{equation*}
    q_i':(t)\mapsto P^{\upT}(q_i+ce)(t),
\end{equation*}
\begin{equation*}
    \veL':(t)\mapsto\veL(t),
\end{equation*}
\begin{equation*}
    \veB':(t)\mapsto\veB^{\tl}(t),
\end{equation*}
\begin{equation}
    \cbR = \partial_t-\diag\left(\nabla_{x'}\cdot\left(A_i'\nabla_{x'}\right)\right)+\diag\left(q_i'\cdot \nabla_{x'}\right) - \veL'.
\end{equation}
By construction, if $\veu'$ is a solution of $\cbR\veu'=-\veB'\veu'\circ\veu'$, 
then $\veu:(t,x)\mapsto\veu'(t,P^{\upT}(x+cte))$ is a solution of $\cbQ\veu=-\veB\veu\circ\veu$.

In the following construction, we only work with the variables $(t,x')$, never with the variables $(t,x)$. 
In order to simplify notations, we will prefer the generic notation $x'=(y,z)\in\R^{n-1}\times\R$, 
\textit{i.e.} $x_\alpha'=y_\alpha$ if $\alpha\in[n-1]$ and $x_n'=z$. 
The notation $\nabla$ should be unambiguously understood as $\nabla_{x'}=\nabla_{(y,z)}$ 
in this construction. 
Also, in this coordinate system, $e$ becomes $e'=P^{\upT}e=(0,0,\dots,0,1)^{\upT}$,
$e\cdot\nabla_x$ becomes $e'\cdot\nabla=\partial_{x'_n}=\partial_z$ and $\nabla_x\cdot(A_i e)$ becomes 
$\nabla\cdot(A_i'e')$.

The family $(A_i')_{i\in[N]}$ of diffusivity matrices in the new coordinate system remains uniformly
elliptic, \textit{i.e.} it still satisfies \ref{ass:ellipticity}.
For any $\mu\geq 0$, the operator $\cbQ_{\mu}$ defined in \eqref{eq:Qmue} becomes 
\begin{equation*}
    \cbR_\mu = \cbR - \diag\left( 2\mu A_i'e'\cdot\nabla \right) - \diag\left(\mu^2 e'\cdot A_i' e'+\mu\nabla\cdot(A_i' e')-\mu q_i'\cdot e' \right).
\end{equation*}

The main interest of these changes of variables and coefficients is outlined in the following key proposition.
Recall that space and time periodicities are implicitly space homogeneity and $1$-periodicity.
\end{quote}

Everything else, up to the end of the proof of Theorem \ref{thm:supercritical_existence}, is unchanged 
-- with of course many notations and precautions that are now superfluous, in particular regarding dependencies with respect to $y$.

Next, at the beginning of the proof of Theorem \ref{thm:critical_existence}, the first sentence is replaced as follows:
\begin{quote}
In this section, we assume again $\lambda_{1,\upp}<0$, we fix $e\in\Sn$ and we prove the existence of a pulsating traveling wave
with speed $c_e^\star$.    
\end{quote}
and everything else up to the end of the proof is unchanged again.

\end{proof}}

\section{Acknowledgements}

L. G. acknowledges support from the ANR via the project Indyana under grant agreement ANR-21-CE40-0008.
L. G. and G. N. acknowledge support from the ANR via the project Reach under grant agreement ANR-23-CE40-0023-01. 
They thank the anonymous referees for valuable reports.

\bibliographystyle{plain}
\bibliography{ref}

\end{document}